\documentclass[10pt, regno]{amsart}
\usepackage[left=0.9in, right=0.9in, top=1.5in, bottom=1.5in]{geometry}
\usepackage[utf8]{inputenc}
\usepackage[T1]{fontenc} 
\usepackage[english]{babel}
\usepackage{yhmath}
\usepackage{stmaryrd}
\usepackage{amsmath}
\usepackage{amsthm}
\usepackage{url}
\usepackage{tikz}
\usepackage{verbatim}
\DeclareMathOperator{\arcsinc}{arcsinc}
\usepackage{graphicx}
\usepackage{enumitem}
\usepackage{footnote} %
\makesavenoteenv{tabular}%
\makesavenoteenv{table}%
\makeatletter
\DeclareFontFamily{U}{tipa}{}
\DeclareFontShape{U}{tipa}{m}{n}{<->tipa10}{}
\newcommand{\arc@char}{{\usefont{U}{tipa}{m}{n}\symbol{62}}}%
\usetikzlibrary{positioning,matrix,arrows,decorations.pathmorphing, patterns, math, intersections, calc}

\newcommand{\arc}[1]{\mathpalette\arc@arc{#1}}

\newcommand{\arc@arc}[2]{%
  \sbox0{$\m@th#1#2$}%
  \vbox{
    \hbox{\resizebox{\wd0}{\height}{\arc@char}}
    \nointerlineskip
    \box0
  }%
}
\makeatother

\newtheorem{theorem}{Theorem}[section]
\newtheorem{conjecture}{Conjecture}

\newtheorem{lemma}[theorem]{Lemma}

\newtheorem{proposition}[theorem]{Proposition}
\newtheorem{definition}[theorem]{Definition}
\newtheorem{remark}[theorem]{Remark}

\newcommand{\R}{\mathbb{R}}

\newcommand{\N}{\mathbb{N}}

\renewcommand{\S}{\mathcal{S}}

\newcommand{\C}{\mathcal{C}}
\newcommand{\K}{\mathcal{K}}

\newcommand{\D}{\mathcal{D}}
\newcommand{\Om}{\Omega}

\newcommand{\rmnote}[1]{}

\usepackage{braket}
\usepackage{esint}
\newcommand{\abs}[1]{{\left|#1\right|}}
\newcommand{\norma}[1]{{\left\Vert#1\right\Vert}}

\newcounter{mnotecount}[section]

\title{Sharp inequalities involving the Cheeger constant of planar convex sets}
\author{Ilias Ftouhi, Alba Lia Masiello, Gloria Paoli}
\date{\today}

\address[Ilias Ftouhi]{Friedrich-Alexander-Universit{\"a}t Erlangen-N{\"u}rnberg, Department of Mathematics, Chair of Applied Analysis (Alexander von Humboldt Professorship), Cauerstr. 11, 91058 Erlangen, Germany.}
\email{ilias.ftouhi@fau.de}

\address[Alba Lia Masiello (Corresponding Author)]{Dipartimento di Matematica e Applicazioni "R. Caccioppoli", Universita degli Studi di Napoli "Federico II", Complesso Universitario Monte S. Angelo, via Cintia - 80126 Napoli, Italy.}
\email{albalia.masiello@unina.it}

\address[Gloria Paoli]{Dipartimento di Matematica e Applicazioni "R. Caccioppoli", Universita degli Studi di Napoli "Federico II", Complesso Universitario Monte S. Angelo, via Cintia - 80126 Napoli, Italy.}
\email{gloria.paoli@unina.it}

\begin{document}

\maketitle
\begin{abstract}
We are interested in finding sharp bounds for the Cheeger constant $h$ via different geometrical quantities, namely the area $|\cdot|$, the perimeter $P$, the inradius $r$, the circumradius $R$, the minimal width $\omega$ and the diameter $d$. We provide new sharp inequalities between these quantities for planar convex bodies and enounce new conjectures based on numerical simulations.  In particular, we completely solve the Blaschke-Santal\'o diagrams describing all the possible inequalities involving the triplets $(P,h,r)$, $(d,h,r)$ and $(R,h,r)$ and describe some parts of the boundaries of the diagrams of the triplets $(\omega,h,d)$, $(\omega,h,R)$, $(\omega,h,P)$, $(\omega,h,|\cdot|)$, $(R,h,d)$ and $(\omega,h,r)$.  

\vspace{1mm}
\end{abstract}

\section{Introduction}
Let $\Om$ be a bounded subset of $\R^2$. The Cheeger constant of $\Om$, introduced in  and by Jeff Cheeger in \cite{cheger} in connection with the first eigenvalue of the Laplacian, is defined as 
\begin{equation}
    \label{chee}
    h(\Om):=\inf\left\{\frac{P(E)}{\abs{E}} \,  : \, E \, \text{measurable, } \, E\subseteq\Om, \, \abs{E}>0\right\},
\end{equation}
where $P(E)$ is the perimeter of $E$ in the sense of De Giorgi and $\abs{E}$ is the area of $E$. 
The minimum in \eqref{chee} is achieved when $\Om$ has Lipschitz boundary, see as a reference \cite{parini}, and the set $E$ that realizes this minimum is called a \emph{Cheeger set} of $\Om$. For the properties of the Cheeger constant and for an introductory survey, see for example \cite{caselles,kawohl,parini}. In particular, in the case of convex sets, the authors in \cite{caselles} prove that the Cheeger set is unique and, in this case, we will denote it by $C_\Om$. At last, a complete characterization of the Cheeger sets of planar convex bodies is provided  in \cite{kawohl}.

The problem of finding the Cheeger constant of a domain has been widely considered and has several applications (see \cite{parini} for a general overview). One of the possible interpretations of the Cheeger constant can be found for instance in the context of maximal flow and minimal cut problems (see \cite{strang}) which has applications in medical images processing (see \cite{appleton}). The Cheeger problem appears also in the study of plate failure under stress (see \cite{keller}).
It is then useful to have estimates of the Cheeger constant in terms of geometric quantities that can be easily computed. 

In the present paper, we are interested in describing all the possible inequalities involving the Cheeger constant of a given compact and convex set $\Om\subset\R^2$ with nonempty interior and the following geometrical quantities: the area $\abs{\Om}$, the perimeter $P(\Om)$, the inradius $r(\Om)$, the circumradius $R(\Om)$, the minimal width $\omega(\Om)$ and the diameter $d(\Om)$. We are then  aiming to study the Blaschke--Santal\'o diagrams involving those functionals and collect them all in one single paper together with new conjectures. 

A Blaschke--Santal\'o diagram is a tool that allows to visualize all the possible inequalities between three geometric quantities. More precisely, we consider three homogeneous shape functionals $(J_1,J_2,J_3)$, that is to say that for every $i\in\{1,2,3\}$ there exists $\alpha_i\in \R$ such that $J_i(t\Omega)=t^{\alpha_i}J_i(\Om)$ for every $t>0$, and  we want to find a system of inequalities describing the set
$$\{(J_1(\Om), J_2(\Om))|\  J_3(\Om)=1, \, \Om \in \mathcal{K}^2\},$$
where we denote by $\mathcal{K}^2$ the class of planar, compact and convex sets with nonempty interior. 

This kind of diagram was introduced by Blaschke in \cite{blaschke2}, in order to investigate all the possible relations between the volume, the surface area and the integral mean curvature in the class of compact convex sets in $\R^3$. 
Following the idea of Blaschke, Santal\'o in \cite{santalo} proposed the study of these diagrams for all the triplets of the following geometrical quantities: area, perimeter, inradius, circumradius, minimal width and diameter. These diagrams were studied for the class of convex sets and six of them are still not completely solved. We refer to the introduction in \cite{delyon2} for an accurate and updated state of art. 
Moreover, for classical results about Blaschke--Santal\'o diagrams, we refer for example to \cite{cifre3,cifre4,cifre2,cifre_salinas, cifre, cifre_gomis,santalo} and for more recent results, we provide the following non-exhaustive list of works \cite{MR3653891,branden,delyon2,bs-numerics,MR4431498, FL21, LZ}. 

In \cite{ftJMAA} and \cite{ftouhi_cheeger} the author studies Blaschke--Santal\'o diagram involving the Cheeger constant. 
More precisely, in \cite{ftJMAA}, it the Blaschke--Santal\'o diagram involving the Cheeger constant, the area and the inradius is fully characterized. It is proved that, if $\Om$ in $\mathcal{K}^2$, then
\begin{equation}\label{eq:hra}
    \frac{1}{r(\Omega)}+\frac{\pi r(\Omega)}{|\Omega|}\leq h(\Omega) \leq \frac{1}{r(\Omega)}+\sqrt{\frac{\pi}{|\Omega|}}, 
    \end{equation}
where the upper bound in \eqref{eq:hra} is achieved by (and only by) sets that are homothetic to their form bodies (see Definition \ref{formbody}), for instance, sets that are homothetic to their form bodies, meanwhile the lower one is achieved by (and only by) stadiums. On the other hand, in \cite{ftouhi_cheeger}, the diagram involving the Cheeger constant, the area and the perimeter is fully characterized and it is proved that if $\Om\in\mathcal{K}^2$, then
  \begin{equation}\label{eq:hpa}
    \frac{P(\Omega)+\sqrt{4\pi|\Omega|}}{2|\Omega|}\leq h(\Omega)\leq \frac{P(\Omega)}{|\Omega|},    
    \end{equation}
where the upper bound is achieved by any set that is Cheeger of itself (in particular stadiums), meanwhile the lower one is achieved, for example, by circumscribed polygons. 

Now, let us state the main results of the paper.  In order to do so, we need to define the following classes of admissible sets (we refer to \cite[Table 2.1]{inequalities_convex} for the associated constraints):
\vspace{2mm}
\begin{enumerate}
    \item $\displaystyle{\mathcal{K}^2_{P,r}=\{\Om \in \mathcal{K}^2: \,  P(\Om)=P, \, r(\Om)=r \}}$, where $P\ge 2\pi r$;
    \vspace{2mm}
    \item $\mathcal{K}^2_{d,r}=\{\Om \in \mathcal{K}^2: \,  d(\Om)=d,  \, r(\Om)=r\}$, where $d\ge 2 r$;
    \vspace{2mm}
    \item $\mathcal{K}^2_{R,r}=\{\Om \in \mathcal{K}^2:  \, R(\Om)=R, \, r(\Om)=r\}$, where $R\ge r$;
    \vspace{2mm}
     \item $\mathcal{K}^2_{\omega,d}=\{\Om \in \mathcal{K}^2: \, \omega(\Om)=\omega, \, d(\Om)=d\}$, where $\omega\le d$;
    \vspace{2mm}
       \item $\mathcal{K}^2_{\omega,R}=\{\Om \in \mathcal{K}^2:  \, \omega(\Om)=\omega, \,  R(\Om)=R\}$, where $2R\ge \omega$;
    \vspace{2mm}
    
    \item $\mathcal{K}^2_{\omega,P}=\{\Om \in \mathcal{K}^2: \, \omega(\Om)=\omega, \, P(\Om)=P\}$, where $P\ge \pi \omega$;
    \vspace{2mm}
    \item $\mathcal{K}^2_{\omega,A}=\{\Om \in \mathcal{K}^2:  \, \omega(\Om)=\omega, \, \abs{\Om}=A\}$, where $\sqrt{3}A\ge \omega^2$;
    \vspace{2mm}
    \item $\mathcal{K}^2_{R,d}=\{\Om \in \mathcal{K}^2:  \, R(\Om)=R, \, d(\Om)=d\}$, where $\sqrt{3}R\le d< 2R$;
    \vspace{2mm}
    \item $\mathcal{K}^2_{\omega,r}=\{\Om \in \mathcal{K}^2: \, \omega(\Om)=\omega, \, r(\Om)=r\}$, where $2r< \omega\le 3r$;
    \vspace{2mm}
    \item $\mathcal{K}^2_{R,A}=\{\Om \in \mathcal{K}^2: \, R(\Om)=R, \, \abs{\Om}=A\}$, where $A\le \pi R^2$;
    \vspace{2mm}
    \item $\mathcal{K}^2_{P,R}=\{\Om \in \mathcal{K}^2:  \, P(\Om)=P, \, R(\Om)=R\}$, where $4 R<P\le 2\pi R$;
    \vspace{2mm}
     \item $\mathcal{K}^2_{P, d}=\{\Om \in \mathcal{K}^2: \,  P(\Om)=P, \, d(\Om)=d\}$, where $2 d<P\leq \pi d$;
  \vspace{2mm}
    \item $\mathcal{K}^2_{d,A}=\{\Om \in \mathcal{K}^2: \, d(\Om)=d, \, \abs{\Om}=A\}$, where $\pi d^2\ge 4A$.
\end{enumerate}
Firstly, let us state the following existence result.
\begin{theorem}\label{th:existence}
Let $\Om\in \mathcal{K}^2$, then the minimization and the maximization shape optimization problems of the Cheeger constant $h(\Om)$ admit a solution in the classes of sets defined in $(1)-(13)$. 
 \end{theorem}

In the following theorem, we consider the triplets of functionals for which we are able to provide the complete description of the related Blaschke--Santaló diagrams. For the precise definitions of the below-mentioned extremal sets, see Section \ref{ext} and for the explicit bounds, see Propositions \ref{prop_hrP}, \ref{prop_hdr} and \ref{prop_hRr}. At last, for the description of the corresponding diagrams we refer to Proposition \ref{blaschke}. 
\begin{theorem}\label{th2}
The following results hold
\begin{enumerate}[label=(\roman*)] 
\item The maximum and the minimum of the Cheeger constant in $\mathcal{K}^2_{P,r}$ are respectively achieved  by sets that are homothetic to their form bodies and by stadiums.
 \vspace{2mm}
 \item The maximum of the Cheeger constant in $\mathcal{K}^2_{d,r}$ is achieved by symmetrical two-cup bodies. On the other hand, there exists $D_0>0$ such that if $d\geq r D_0$, then the minimum of the Cheeger constant in $\mathcal{K}^2_{d,r}$ is achieved by symmetrical spherical slices, while, if $d<rD_0$, the minimum is achieved by regular smoothed nonagons.
    \vspace{2mm}
 \item The maximum and the minimum of the Cheeger constant in $\mathcal{K}^2_{R,r}$ are respectively achieved  by symmetrical two-cup bodies and symmetrical spherical slices.
\end{enumerate}

\end{theorem}


As far as the classes of sets $\K^2_{\omega,d}$ $\K^2_{R,\omega}$, $\K^2_{\omega, P}$ and $\K^2_{A,\omega}$ are concerned, we are able to identify parts of the boundary of the corresponding Blasche- Santal\'o diagrams, see Propositions \ref{prop_hdw}, \ref{prop_hRw}, \ref{prop_hwP} and \ref{prop_hAw} for the explicit bounds. For the class  $\mathcal{K}^2_{d, R}$, we are able to solve the maximization problem, see Proposition \ref{prop_hRd}, meanwhile, for the class $\mathcal{K}^2_{\omega,r}$ we are able to solve the minimization one, see Proposition \ref{prop_hwr}. Throughout the paper, different strategies of proofs are used to obtain all the aforementioned results. 

As far as the classes $\K^2_{A,R}, \K^2_{R,P}, \K^2_{P,d}$ and $\K^2_{A,d}$ are concerned, we are not able to identify any parts of the boundaries of the corresponding Blasche--Santal\'o diagrams. Nevertheless, in the appendix we present the best bounds that we managed to obtain by combining the known inequalities involving those functionals. 

\vspace{3mm}

The paper is organized as follows: in Section \ref{sec2}, we state the preliminary results, the definitions used throughout the paper and the known inequalities relating the Cheeger constant to one of the geometric quantities taken into consideration. Section \ref{secnum} is dedicated to the description of the numerical methods used to compute the functionals and to approximate the Blaschke--Santal\'o diagrams. In Section \ref{sec3}, we prove the main results, that are  Theorem \ref{th:existence} and Theorem \ref{th2}. Section \ref{partial} is dedicated to the results that we have obtained for the classes of sets $\K^2_{\omega,d}$ $\K^2_{R,\omega}$, $\K^2_{\omega, P}$, $\K^2_{A,\omega}$, $\mathcal{K}^2_{d, R}$ and $ \mathcal{K}^2_{\omega,r}$.
Finally, in Section \ref{secult}, we state some relevant conjectures and collect all the inequalities proved in the paper in the Appendix. 

\section{Notations and Preliminaries}\label{sec2}
Throughout this article, $\norma{\cdot}$ will denote the Euclidean norm in $\mathbb{R}^2$,
 while $(\cdot)$ is the standard Euclidean scalar product in $\R^2$. We denote by $P(\Om)$ the perimeter of $\Omega$ and by $|\Om|$ the volume of $\Om$. Moreover, $B_r$ is the closed ball of radius $r>0$ centered at the origin, while $\mathbb{S}^1$ is the unit sphere in $\R^2$. In the following, we work with the class of sets $\mathcal{K}^2$, defined as 
\begin{equation*}
    \mathcal{K}^2:=\{ \Om\:|\: \Om \;\, \text{is a compact, bounded and convex set with nonempty interior of  }\; \R^2\}\setminus \{\emptyset\}.
\end{equation*}

\subsection{Classical results ad preliminary lemmas }

We provide the classical definitions and results that we need in the following.
\begin{definition}
    \label{minksum}
    Let $\Omega, K\subset \R^2$ two convex bounded sets. We define the \emph{Minkowski sum} $(+)$ and \emph{difference} $(\sim)$ as
    \begin{equation*}
        \label{sum}
        \Omega+K:=\{x+y \, : \, x\in \Omega, \, y\in K\},
    \end{equation*}
    \begin{equation*}
        \label{diff}
        \Omega\sim K:=\{x\in \R^2 \, : \, x+K\subseteq \Omega\}.
    \end{equation*}
\end{definition}

We now recall the definition of the Hausdorff distance.
\begin{definition}
Let $\Omega,K\subset\mathbb{R}^2$ be two non-empty compact sets, we define the Hausdorff distance between $\Om$ and $K$ as follows:
\begin{equation*}
\label{disth}
 d_{\mathcal{H}}(\Omega,K):=\inf \left\{  \varepsilon>0  \; :\; \Omega\subset K+B_{\varepsilon}, \; K\subset\Omega+B_{\varepsilon} \right\},
 \end{equation*}
 where $B_\varepsilon$ is the ball of radius $\varepsilon$ centered in the origin. 
\end{definition}
 
Let $\{\Omega_k\}_{k\in\N}$ be a sequence of non-empty, compact, bounded convex subsets of $\R^2$, we say that $\Omega_k$ converges to $\Om$ in the Hausdorff sense and we denote
\[
\Omega_k\stackrel{\mathcal H}{\longrightarrow} \Omega,
\]
if and only if $d_{\mathcal H}(\Omega_k,\Omega)\to 0$ as $k\to \infty$. 

We recall that by Blaschke's selection Theorem (see for example \cite[Theorem 1.8.7]{schneider}), every bounded sequence of nonempty compact convex sets has a subsequence that converges in the Hausdorff sense to a convex set.

Let us now recall the following definitions:
\begin{definition}
  Let $\Om\in\mathcal{K}^2$. The \emph{distance function from the boundary of} $\Om$ is the function $ {\rm dist}(\cdot, \partial \Omega):\Omega \to [0,+\infty[$ defined as
 $${\rm dist}(x,\partial\Omega)=\inf_{y\in\partial\Omega}\norma{x-y}.$$
  The \emph{inradius} of $\Omega$  is defined as	$$r(\Omega):=\sup_{x\in \Omega} {\rm dist}(x,\partial\Omega),$$
    and the \emph{circumradius} of $\Om$ is defined as
    $$R(\Om):= \min_{x\in\Om}\max_{y\in\partial\Om} \norma{x-y}.$$
   \end{definition}
Let us now introduce the support function of a convex set:
\begin{definition}\label{support}
Let $\Omega\in \mathcal{K}^2$. The \emph{support} function of $\Omega$ is defined as
    \begin{equation*}
        p_\Omega(y):=\max_{x\in \Omega} (x\cdot y), \qquad y\in \mathbb{R}^2.
    \end{equation*}
\end{definition}

 In this paper, we will also consider the minimal width (or thickness) of a convex set, that is to say the minimal distance between two parallel supporting hyperplanes. More precisely, we have 
\begin{definition}\label{width:def}
    Let $\Omega\in\mathcal{K}^2$. The width of $\Omega$ in the direction $y \in \mathbb{S}^1$ is defined as 
    \begin{equation*}
        \omega_{\Omega}(y):=p_\Omega(y)+p_\Omega( -y)
    \end{equation*}
    and the \emph{minimal width} of $\Omega$ as
\begin{equation*}
    \omega(\Omega):=\min\{  \omega_{\Omega}(y)\,|\; y\in\mathbb{S}^{1}\}.
\end{equation*}
\end{definition}

We introduce the inner parallel set of a convex set $\Omega$.                                  
\begin{definition}\label{def:inner_parallel}
    \label{inner_parallel}
    Let $\Omega$ be a bounded and convex set. The \emph{inner parallel set} of $\Omega$ at distance $t\in [0, r(\Om)]$ is
    \begin{equation*}
        \Omega_{-t}:=\{x\in \Omega \, : \, {\rm dist}(x,\partial\Omega)\ge t\}.
    \end{equation*}
\end{definition}

 \begin{remark}\label{rem:in}
 We remark that, by Definition \ref{def:inner_parallel}, we have 
$$\Om_{-t}=\Om \sim tB_1.$$
Moreover, we observe that for any $y\in \mathbb{S}^1$ and for every $\Om,K\in \mathcal{K}^2$ , one has
 \begin{equation*}
     p_{\Om\sim K}( y)\le p_\Om(y)-p_K(y),
 \end{equation*}
 see e.g. \cite[Section 3.1, page 148]{schneider}. 
 
 Therefore, in the case $K=tB_1$, this reads
 \begin{equation}\label{diff_parall}
     p_{\Om_{-t}}(y)\le p_\Om(y)-t.
 \end{equation}
 Moreover, as it is observed in \cite[Proposition 3.2]{jahn}, 
 one has
\begin{equation}
    \label{circpar}
    R(\Om+ K)\le R(\Om)+R(K)
\end{equation}
with equality if $K=tB_1$.
\end{remark}

We are now in position to prove the following lemma:
\begin{lemma}\label{lem:diameter_inner_set}
Let $\Om\in\mathcal{K}^2$. We have for every $ t\in [0,r(\Om)]$: 
\begin{align}
\label{inr}
     &r(\Om_{-t})= r(\Om)-t, \\ 
    \label{eq:diameter}
     &d(\Om_{-t})\leq d(\Om)-2t,\\
     \label{width}
     &\omega (\Om_{-t})\leq \omega (\Omega)-2t,\\
     \label{circumradius}
     &R(\Om_{-t})\leq R(\Om)-t,\\
     \label{perimeter}
     &P(\Om_{-t})\leq P(\Om)-2\pi t.
\end{align}
\end{lemma}
\begin{proof}
\begin{itemize}
    \item The proof of \eqref{inr} can be found in \cite[Lemma 1.4]{larson}.
    \item Let us now prove \eqref{eq:diameter}.
Let $x_t,y_t\in \Om_{-t}$ be two diametrical points of $\Om_{-t}$ (i.e., such that $\|x_t-y_t\|=d(\Om_{-t})$). We denote by $x,y\in \Om$ the points corresponding to the intersection of the line containing $x_t$ and $y_t$ with the boundary of $\Om$. We have
$$d(\Om)\ge \|x-y\|=\|x-x_t\|+\|x_t-y_t\|+\|y_t-y\|=\|x-x_t\|+d(\Om_{-t})+\|y_t-y\|\ge d(\Om_{-t})+2t,$$
where the last inequality is a consequence of the fact that $x_t,y_t\in \Om_{-t}=\{x\in \Om\ |\ {\rm dist}(x,\partial \Om)\ge t\}$. 
\item The proof of \eqref{width} follows directly from the definition of the minimal width (Definition \ref{width:def}) and \eqref{diff_parall}.
\item We prove now \eqref{circumradius}. 
 As observed in Remark \ref{rem:in}, and, in particular, by formula \eqref{circpar}, for every $\Om\in \mathcal{K}^2$, we have that $R(\Om+tB_1)=R(\Om)+t$. Thus, we have $$R(\Omega_{-t})= R(\Omega_{-t}+tB_1)-t\leq R(\Omega)-t.$$ 
  The last inequality follows from the inclusion $\Omega_{-t}+t B_1 \subset \Omega$ and the monotonicity of the circumradius with respect to inclusions. 

  \item Formula \eqref{perimeter} can be obtained as a consequence of the classical Steiner formula  
$$P(K+tB_1)=P(K)+2\pi t,$$
see for example \cite[Section 4.1]{schneider}, and the fact that the perimeter is monotone with respect to the inclusion for convex sets. Indeed, since $\Omega_{-t}+t B_1 \subset \Omega$, we have $P(\Omega_{-t}+t B_1 ) \leq P(\Omega)$, which is equivalent  by the Steiner formula to $P(\Omega_{-t})+2\pi t\leq P(\Omega)$. 
\end{itemize}

\end{proof}

\newpage
The following Lemma will play a key role in the proof of Theorem \ref{th2}. 
\begin{lemma}\label{lem:main}
Let $\Om\in \K^2$. We assume that there exists a continuous function $g^\Om:[0,r(\Om)]\rightarrow \R$ such that 
\begin{equation}\label{kawohl}
  \forall t\in [0,r(\Om)],\ \ \ \ |\Om_{-t}|\leq g^\Om(t),\ \ \ \text{(resp. $|\Om_{-t}|\ge g^\Om(t)$)}, 
\end{equation}
and that
\begin{equation}\label{set}
   G_{\Omega}:= \Set{t\in [0,r(\Om)] : g^\Om(t)=\pi t^2}\neq \emptyset.
\end{equation}
 We have 
\begin{equation}\label{kawohl2}
    h(\Om)\ge \frac{1}{t_{g^\Om}}\ \ \ \ \ \text{(resp. $h(\Om)\leq \frac{1}{t_{g^\Om}}$)},
\end{equation}
where $t_{g^\Om}$ is the smallest (resp. largest) solution to the equation $g^\Om(t) = \pi t^2$ on $[0,r(\Om)]$. 
\end{lemma}
\begin{proof}
From  \cite[Theorem 1]{kawohl}, we know that there exists a unique $t=t_\Om>0$ such that $\abs{\Om_{-t}}=\pi t^2$ and $\displaystyle{h(\Om)=1/t_\Om}$. It is then clear that, if there exists a function $g(t)$ such that \eqref{kawohl} and \eqref{set} hold, then, the smallest (resp. largest) element $t_{g^\Om}\in G_\Om$ must satisfy \eqref{kawohl2} (see Figure \ref{ilias}).
\end{proof}

\begin{figure}[h]
    \centering
    \includegraphics[scale=.6]{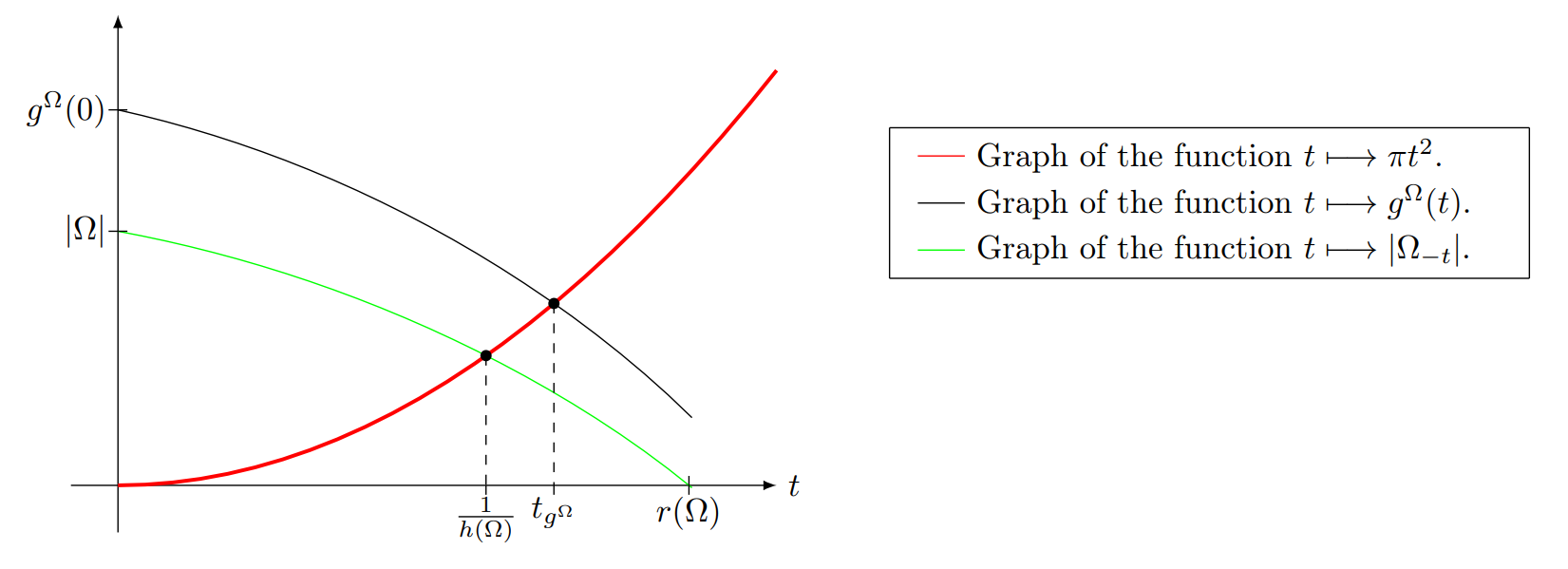}
    \caption{The idea of the proof of  Lemma \ref{lem:main}.}
    \label{ilias}
\end{figure}

\subsection{Extremal sets and their properties} \label{ext}
In this Section, we describe special planar shapes that appear in the statement of the main results.  

Firstly, let us recall the definition of the form body of a convex set $\Om$, following \cite{schneider}: a point $x\in\partial\Omega$ is called \emph{regular} if the supporting hyperplane at $x$ is uniquely defined. The set of all regular points of $\partial\Om$ is denoted by ${\rm reg}(\Om)$. 
We also let $U(\Om)$ denote the set of all outward pointing unit normals to $\partial\Om$ at points of ${\rm reg}(\Om)$.

\begin{definition}\label{formbody}
The form body $\Omega^\star$ of a set $\Om\in\mathcal{K}^2$ is
defined as
$$\Om^\star =\bigcap_{u\in U(\Om)} \{x\in \R^2: \, (x,u)\le 1\}.$$
\end{definition}

Convex sets that are homothetic to their form bodies will appear in the following as extremal sets. In particular, a polygon whose incircle touches all its sides is homothetic to its form body.

\newpage
\begin{definition}\label{stadium}
A \emph{stadium} $\mathcal{R}$ is defined as the convex hull of the union of two balls in $\R^2$ with the same radius, see Figure \ref{fig:stad}.
\end{definition}

\begin{figure}[h]
    \centering
    \includegraphics[scale=.32]{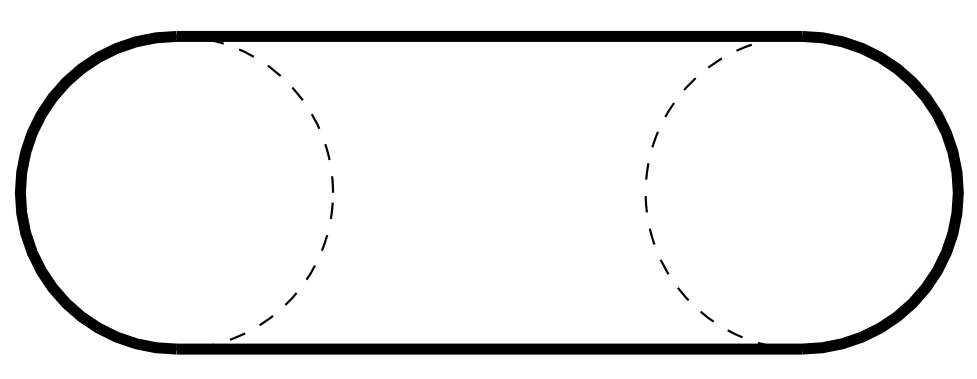}
    \caption{A stadium.}
    \label{fig:stad}
\end{figure}

\begin{definition}\label{def:slice}
The \emph{symmetrical spherical slice} $\mathcal{S}$ of diameter $d$ and width $\omega\leq d$ is the convex set obtained by the intersection of a ball of radius $d/2$ and a strip of width $\omega$ centered at the center of the ball, see Figure \ref{fiig:slice}.
\end{definition}

\begin{figure}[h]
    \centering
    \includegraphics[scale=.32]{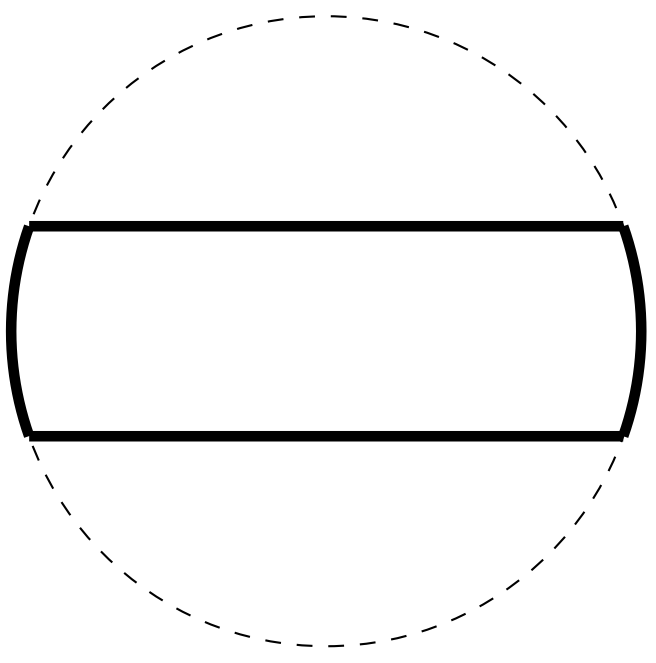}
    \caption{A symmetrical spherical slice.}
    \label{fiig:slice}
\end{figure}

\begin{definition}
A \emph{two-cup body} $\mathcal{C}$ is the convex hull of a ball in $\R^2$ with two points that are symmetric with respect to the center of the ball, see Figure \ref{fig:cup}. In particular, a two-cup body is homothetic to its form body.
\end{definition}

\begin{figure}[h]
    \centering
    \includegraphics[scale=.32]{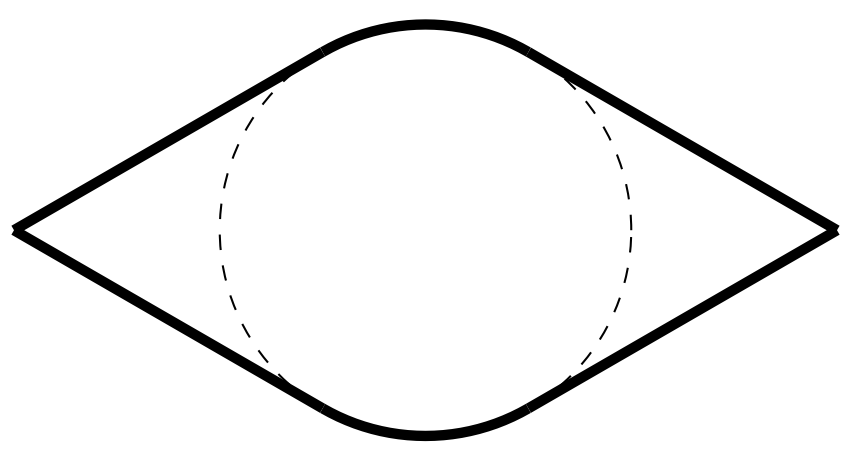}
    \caption{A two-cup body.}
    \label{fig:cup}
\end{figure}

\begin{definition}
A \emph{subequilateral triangle} $T_I$ is an isosceles triangle with two equal angles greater than $\pi/3$.
\end{definition}

The following class of sets (introduced in \cite{yamanouti}, see also \cite{santalo}) represents a way to pass in a continuous manner with respect to the Hausdorff distance from the equilateral triangle to the Reuleaux triangle. 

\begin{definition}\label{yamanouti}
A \emph{Yamanouti set} $Y$ is the convex hull of an equilateral triangle and three circular arcs of equal
radius, whose centers are each of the vertices of this triangle; the radius of these circular arcs
is less than the length of the edge of the triangle, see Figure \ref{yama}. 
\end{definition}

\begin{figure}[h]
    \centering
    \includegraphics[scale=.6]{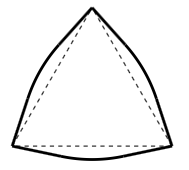}
    \caption{A Yamanouti set.}
    \label{yama}
\end{figure}

\newpage
In \cite{delyon2}, the authors define the smoothed regular nonagon as follows:
\begin{definition}\label{def:nonagon}
Let $r>0$ and $2 r<d<2\sqrt{3} r$. The \emph{smoothed regular nonagon} of inradius $r$ and diameter $d$, that we denote by $\mathcal{N}_{r,d}$, is the convex set enclosed in an equilateral triangle $T_E$ with barycenter in the origin and such that $r(T_E)=r$, obtained following the construction below. 
Let $\eta_i$ the normal angles to the sides of $T_E$ and let
$$\tau:= (3+\sqrt{d^2-3 r^2})/2, \,\, \, \, \text{and}\, \, \, \,  h:=\sqrt{d^2-\tau^2}.$$
We now define the points $A_i, B_i, M_i$, for $i=1,2,3$:
$$A_i:=r \begin{pmatrix}
\cos{\eta_i}+h \sin{\eta_i} \\
\sin{\eta_i}-h\cos{\eta_i}
\end{pmatrix}, \, \, B_i:=r \begin{pmatrix}
\cos{\eta_i}-h \sin{\eta_i} \\
\sin{\eta_i}+h\cos{\eta_i}
\end{pmatrix}, \, \, \, M_i:=r (1-\tau)\begin{pmatrix}
\cos{\eta_i} \\
\sin{\eta_i}
\end{pmatrix}.$$
We obtain $\mathcal{N}_{r,d}$ as follows (see Figure \ref{fig:non}):
\begin{itemize}
    \item the points $A_i, B_i$ and $M_i$, for $i=1,2,3$, belong to $\partial \mathcal{N}$;
    \item $\arc{B_1M_3}$ and $\arc{M_1A_3}$ are diametrically opposed arcs of the same circle of diameter $d$, the same for the pairs $\arc{B_2M_1}$ and $\arc{M_2A_1}$, $\arc{M_2 B_3}$ and $\arc{M_3A_2}$;
    \item the boundary contains the segments $\overline{A_iB_i}$, for $i=1,2,3$, and the contact point $I_i$ with the incircle is the middle of the corresponding segment.
\end{itemize}
\end{definition}


\begin{figure}[h]
    \centering
    \includegraphics[scale=.4]{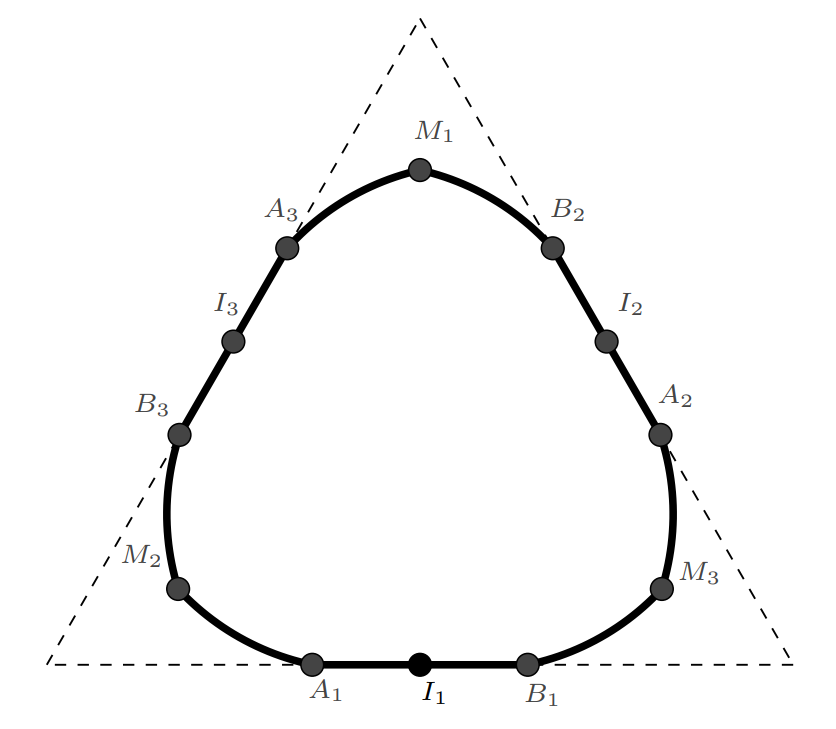}
    \caption{A smoothed regular nonagon.}
    \label{fig:non}
\end{figure}

Let us now recall some sharp inequalities that we will need in the sequel. Those classic estimates are obtained via the study of Blaschke--Santal\'o diagrams involving the following geometric quantities: the perimeter, the area, the inradius, the circumradius, the diameter and the minimal width. 

Firstly, let us consider the diagram $(A,d,r)$. We have the following two theorems.

\begin{theorem}[\cite{cifre_salinas} and \cite{delyon2}]
Let $\Om\in \mathcal{K}^2$. We have 
\begin{equation}
    \label{del_2cap}
    \abs{\Om}\ge r(\Om)\sqrt{d^2(\Om)-4r^2(\Om)}+r^2(\Om)\left(\pi- 2\arccos{\left(\frac{2r(\Om)}{d(\Om)}\right)}\right),
\end{equation}
where the equality holds if and only if $\Om$ is a two-cup body. 
\end{theorem}

\begin{theorem}[\cite{delyon2}, Theorem 2] \label{delthm}
Let $\Om\in \mathcal{K}^2$. We have
\begin{equation}\label{del:non}
     \abs{\Om}\le \psi\left(d\left(\Om\right),r(\Om)\right), 
\end{equation}
where
\begin{equation}\label{func_del}
   \psi(d,r):=\begin{cases}
    \displaystyle{\frac{3\sqrt{3}r}{2}(\sqrt{d^2-3r^2}-r)+\frac{3d^2}{2}\left(\frac{\pi}{3}-\arccos{\left(\frac{\sqrt{3}r}{d}\right)}\right)}, & \text{if} \, \, \,  d\le r D^*\vspace{1mm} \\
    \displaystyle{r\sqrt{d^2-4r^2}+\frac{d^2}{2}\arcsin{\left(\frac{2r}{d}\right)}}, & \text{if} \, \, \, d\ge r D^*
\end{cases} 
\end{equation}
and $D^*$ is the unique number in $[2,2\sqrt{3}]$ for which the two expressions of the function $\psi(d,1)$ are equal, i.e., 
\begin{equation*}
    \displaystyle{\frac{3\sqrt{3}}{2}(\sqrt{(D^{\ast})^2-3}-1)+\frac{3(D^\ast)^2}{2}\left(\frac{\pi}{3}-\arccos{\left(\frac{\sqrt{3}}{D^\ast}\right)}\right)}=\displaystyle{\sqrt{(D^{\ast})^2-4}+\frac{(D^{\ast})^2}{2}\arcsin{\left(\frac{2}{D^\ast}\right)}}.
\end{equation*}
Moreover, if $d(\Om)\le r(\Om)D^*$, we have the equality in \eqref{del:non} if and only if $\Om$ is a regular smoothed nonagon, while, if $d(\Om)>r(\Om)D^*$, we have the equality if and only if $\Om$ is a symmetrical spherical slice.
\end{theorem}

As far as the diagram $(A, \omega, R)$ is concerned, we recall the following theorems: 
\begin{theorem}[\cite{cifre_salinas}, Theorem 3] \label{cifrsalth}
Let $\Om\in\mathcal{K}^2$. Then, it holds
\begin{equation}
\label{henk}
\abs{\Om}\leq \chi(\omega(\Om), R(\Om)),
\end{equation}
where
\begin{equation}\label{henkchi}
\chi(\omega(\Om), R(\Om)):=\frac{\omega(\Om)}{2} \sqrt{4 R(\Om)^2-\omega(\Om)^2} + 2R(\Om)^2 \arcsin{\frac{\omega(\Om)}{2R(\Om)}},
\end{equation}
and the equality in \eqref{henk} holds if and only if $\Om$ is a symmetrical spherical slice. 
\end{theorem}

\begin{theorem}[\cite{cifre_salinas}, Theorem 6]
Let $\Om\in\mathcal{K}^2$. Then, if $\omega(\Om)\leq \frac{3}{2}R(\Om)$, it holds
\begin{equation}\label{ARw_low}
   16\abs{\Om}^6\geq R^2(\Om) \omega^2(\Om)\left(16 \abs{\Om}^4-R^2(\Om)\omega^6(\Om)\right)
\end{equation}
and the equality holds if and only if $\Om$ is a subequilateral triangle.
\end{theorem}

We recall the following inequality from the diagram $(R,r,\omega)$.

\begin{theorem}[\cite{cifre_gomis}, Theorem 2]\label{cifre_gomis_1_thm}
Let $\Om\in\mathcal{K}^2$. Then, it holds
\begin{equation}\label{cifre_gomis_1_eq}
    \left(4 r(\Om)-\omega(\Om)   \right)\left( \omega(\Om)-2 r(\Om)\right)\leq \frac{2 r^3(\Om)}{R(\Om)}
\end{equation}
and the equality holds if and only if $\Om$ is an isosceles triangle.
\end{theorem}

\begin{theorem}[\cite{santalo}, Section 10]\label{santalo_thm_1}
Let $\Om\in\mathcal{K}^2$. Then, it holds
\begin{equation}\label{santalo_eq_1}
    \omega(\Omega)\leq R(\Omega)+r(\Omega),
\end{equation}
where the equality is achieved by any set of constant width (i.e., the set $\Om$ is such that the function $\omega_\Om$, defined in Definition \ref{width:def}, is constant).

\end{theorem}

The following theorem deals with the $(A, r, R)$ diagram.

\begin{theorem}[\cite{cifre_salinas}, Theorems 1 and 2]\label{inr_circ_thm}
Let $\Om\in\mathcal{K}^2$. Then, it holds
\begin{equation}\label{secondhRr}
    \abs{\Om}\ge 2r(\Om)\left(\sqrt{R(\Om)^2-r(\Om)^2}+r\arcsin{\frac{r(\Om)}{R(\Om)}}\right), 
\end{equation}
and the equality in \eqref{secondhRr} holds if and only if $\Om$ is a two-cup body. 
Moreover, we have
\begin{equation}\label{cifreRr}
    \abs{\Om}\le \varphi(R(\Om), r(\Om)),
\end{equation}
where
\begin{equation}
    \label{cifreArR}
    \varphi(R(\Om), r(\Om)):= 2\left(r\sqrt{R(\Om)^2-r(\Om)^2}+R^2(\Om)\arcsin{\frac{r(\Om)}{R(\Om)}}\right),
\end{equation}
and the equality in \eqref{cifreArR} holds if and only if $\Om$ is a symmetrical spherical slice.
\end{theorem}

Now we quote two inequalities concerning the $(A,\omega,r)$ and $(P,\omega,r)$ diagrams.

\begin{theorem}[\cite{cifre_salinas}, Theorem 5] \label{cifre_salinas_thm_rw}
Let $\Om \in \mathcal{K}^2$. We have
\begin{equation} \label{cifre_salinas_rwA}
    (\omega(\Om)-2r(\Om))^2(4r(\Om)-\omega(\Om))\abs{\Om}^2\le r^4(\Om)\omega^3(\Om)
\end{equation}
and 
\begin{equation} \label{cifre_salinas_rwP}
    (\omega(\Om)-2r(\Om))^2(4r(\Om)-\omega(\Om))P^2(\Om)\le 4 r(\Om)^2\omega^3(\Om)
\end{equation}
In both inequalities, the equality holds if and only if $\Om$ is a subequilateral triangle.
\end{theorem}
Finally, we recall this result from the $(d, \omega,r)$ diagram.
\begin{theorem}[\cite{cifre4}, Theorem 1-2]\label{cifre_wrd_thm}
Let $\Omega\in\mathcal{K}^2$. We have
  \begin{equation}\label{cifre_wrd_eq_1}
        d^2(\Om)(\omega(\Om)-2r(\Om))^2(4r(\Om)-\omega(\Om))\le 4r^4(\Om)\omega(\Om),
    \end{equation}
    where the equality holds if and only if $\Om$ is a subequilateral triangle $T_I$, and 
    \begin{equation}\label{cifre_wrd_eq_2}
        \omega(\Om)-r(\Om)\leq \dfrac{\sqrt{3}}{3} d(\Om),
    \end{equation}
    where the equality holds if $\Om$ is a Yamanouti set.
\end{theorem}

\subsection{Inequalities relating the Cheeger constant to one geometric quantity} 
In the following paragraph, we state the inequalities relating the Cheeger constant to one of the geometric quantities taken into account, obtained by combining classical results.
\begin{proposition}\label{two}
Let $\Om\in\K^2$. We have
\begin{enumerate}
    \item $h(\Om)\ge 2\sqrt{\frac{\pi}{|\Om|}}$.
    \item $h(\Omega)\ge\frac{4\pi}{P(\Om)}$.
    \item $h(\Om)\geq \frac{4}{d(\Om)}$.
    \item $\frac{1}{r(\Om)}\leq h(\Om)\leq \frac{2}{r(\Om)}$.
    \item $h(\Om)\geq \frac{2}{R(\Om)} $. 
\end{enumerate}
In $(1)-(2)-(3)-(5)$, the equality is achieved by balls, while the one in the upper bound in $(4)$ is achieved by balls and the lower bound is asymptotically an equality for thinning vanishing stadiums.
\end{proposition}
\begin{proof}
\begin{itemize}
    \item We have, by using the isoperimetric inequality $P(E)\ge 2\sqrt{\pi|E|}$,
    \begin{equation}\label{h1}
        h(\Om)= \inf_{\underset{E\subset \Omega,\  |E|>0}{E\text{ is measurable }}}\frac{P(E)}{|E|}\ge \inf_{\underset{E\subset \Omega,\  |E|>0}{E\text{ is measurable }}} \frac{2\sqrt{\pi}\sqrt{|E|}}{|E|}= \inf_{\underset{E\subset \Omega,\  |E|>0}{E\text{ is measurable }}} \frac{2\sqrt{\pi}}{\sqrt{|E|}}= \frac{2\sqrt{\pi}}{\sqrt{|\Om|}}.
    \end{equation}
 
    \item By using \eqref{h1} and the isoperimetric inequality, we have
    $$h(\Om)\ge \frac{2\sqrt{\pi}}{\sqrt{|\Om|}} \ge \frac{4\pi}{P(\Om)}. $$
    
    \item By using \eqref{h1} and the isodiametric inequality $|\Om|\leq \frac{\pi}{4}d(\Om)^2$, we have
    $$h(\Om)\ge \frac{2\sqrt{\pi}}{\sqrt{|\Om|}} \ge \frac{2\sqrt{\pi} }{\sqrt{\frac{\pi}{4}d(\Omega)^2}}=\frac{4}{d(\Omega)}.$$
   
    \item For the upper bound, we have 
    $$h(\Omega)= \inf_{\underset{E\subset \Omega,\  |E|>0}{E\text{ is measurable }}}\frac{P(E)}{|E|}\leq \frac{P(B_{r(\Om)})}{|B_{r(\Om)}|}= \frac{2}{r(\Om)},$$
    where $B_{r(\Om)}$ is a ball of radius $r(\Om)$ inscribed in $\Om$. As, for the lower bound, we have 
    $$h(\Om)= \inf_{\underset{E\subset \Omega,\  |E|>0}{E\text{ is measurable }}}\frac{P(E)}{|E|} \ge \inf_{\underset{E\subset \Omega,\  |E|>0}{E\text{ is measurable }}}\frac{1}{r(E)}\ge \frac{1}{r(\Om)},$$
    where we use the inequalities $|E|<r(E)P(E)$ (see \cite{inequalities_convex}) and $r(E)\leq r(\Om)$.
    \vspace{2mm}
    \item By using \eqref{h1} and the inequality $|\Om|\leq \pi R(\Om)^2$ (see \cite{inequalities_convex}), we have 
    $$h(\Om)\ge \frac{2\sqrt{\pi}}{\sqrt{|\Om|}} \ge \frac{2}{R(\Om)}.$$
\end{itemize}
\end{proof}

\section{Numerical results and Blaschke--Santal\'o diagrams}\label{secnum}
In this Section, we introduce numerical tools, that we use to obtain more information on the diagrams and state some conjectures.

\subsection{Generation of random convex polygons}
We want to provide a numerical approximation of the diagrams studied in Section \ref{partial}. To do so, a natural idea is to generate a large number of convex bodies (more precisely convex polygons). Subsequently, for each of these sets, we calculate the involved functionals. Nevertheless, the task of (properly) generating random convex polygons is quite challenging and of intrinsic interest. The main difficulty is that one wants to design an efficient and fast algorithm that allows to obtain a uniform distribution of the generated random convex polygons. For clarity, let us discuss two different (naive) approaches:
\begin{itemize}
    \item one easy way to generate random convex polygons is by rejection sampling. We generate a random set of points in a square; if they form a convex polygon, we return it, otherwise we try again. Unfortunately, the probability of a set of $n$ points uniformly generated inside a given square to be in convex position is equal to $p_n = \left(\frac{\binom{2n-2}{n-1}}{n!}\right)^2$, see \cite{random_polygon}. Thus, the random variable $X_n$ corresponding to the expected number of iterations needed to obtain a convex distribution follows a geometric law of parameter $p_n$, which means that its expectation is given by $\mathbb{E}(X_n)=\frac{1}{p_n}=\left(\frac{n!}{\binom{2n-2}{n-1}}\right)^2 $. For example, if $N=20$, the expected number of iterations is approximately equal to $2.10^9$, and, since one iteration is performed in an average of $0.7$ seconds, this means that the algorithm will need about $50$ years to provide one convex polygon with $20$ sides; 
    \item another natural approach is to generate random points and take their convex hull. This method is quite fast, as one can compute the convex hull of $N$ points in a $\mathcal{O}(N\log(N))$ time (see \cite{MR475616} for example), but it is not quite relevant since it yields to a biased distribution.   
\end{itemize}

In order to avoid the issues stated above, we use an algorithm presented in \cite{sander}, that is based on the work of P. Valtr \cite{random_polygon}, where the author computes 
the probability of a set of $n$ points uniformly generated inside a given square to be in convex position. The author remarks (in Section 4) that the proof yields a fast and 
non-biased method to generate random convex sets inside a given square. We also refer to \cite{sander} for a nice description of the steps of the method and a beautiful 
animation where one can follow each step; one can also find an implementation of Valtr's algorithm in Java that we decided to translate in Matlab.

To obtain an approximation of the Blaschke--Santaló
 diagram, we generate $100.000$ random convex polygons of unit area and number of sides between $3$ and $30$, for which we compute the involved functionals. We then obtain 
clouds of dots that provide approximations of the diagrams. This approach has been used in several works, we refer for example 
to \cite{AH11}, \cite{ftJMAA} and \cite{FL21}. For a new and efficient method of the approximation of Blaschke--Santal\'o diagrams based on centroidal Voronoi tessellations, we refer to the recent work \cite{bogosel2023numerical}. 

\subsection{About the computation of the functionals} 
Let us give few details on the numerical computation of the functionals involved in the paper.
\begin{itemize}
    \item The \textbf{Cheeger constant} is computed by using a code implemented by Beniamin Bogosel in \cite{zbMATH07173414} based on the characterization of the Cheeger sets of planar convex sets given in \cite{kawohl} and the \texttt{Clipper} toolbox, a very good implementation of polygon offset computation by Agnus Johnson.
    \item The \textbf{inradius} is also computed by using the \texttt{Clipper} toolbox and the fact that $r(\Omega)$ is the smallest solution to the equation $|\Om_{-t}|=0$. 
    \item The \textbf{diameter} is computed via a fast method of computation, which consists in finding all antipodal pairs of points and looking for the diametrical one between them. This is classically known as Shamos algorithm \cite{MR805539}. 
    \item The \textbf{area} is computed by using Matlab's function \texttt{"polyarea"}. 
    \item The \textbf{minimum width} of a polygon $\Om$ of vertices $\{A_1,\cdots,A_N\}$ is computed by using the following formula
    $$\omega(\Om)=\min_{i\in \llbracket 1,N \rrbracket} \max_{j\in \llbracket 1,N \rrbracket} {\rm dist}(A_j,(A_iA_{i+1})),$$
    where ${\rm dist}(A_j,(A_iA_{i+1}))$ corresponds to the distance between the point $A_j$ and the line $(A_iA_{i+1})$ (with the convention $A_{N+1}:=A_1$). 
    \item The \textbf{circumradius} of a convex set $\Omega$ can be written as follows
    $$R(\Om) = \min_{c\in \Omega}\max_{x \in \Omega} \|c-x\|.$$
    It is computed by using Matlab's routine \texttt{"fminmax"}. 
\end{itemize}

\section{Proof of the main Results} \label{sec3}
\subsection{Proof of Theorem \ref{th:existence} }
We start this Section by proving the existence results stated in Theorem \ref{th:existence}.

\begin{proof}[Proof of the existence] 
Let us consider the minimization problem of the Cheeger constant in the classes of sets $(1)-(13)$; the maximization problem can be dealt with similarly.

For all of these classes of sets, in order to prove the existence of the solution, we consider a minimizing sequence $(\Omega_k)_{k\in\N}$ and we prove that it satisfies the hypothesis of the Blaschke Selection Theorem (see \cite[Theorem 1.8.7]{schneider}), that is to say, its boundedness up to translations. Concerning the class of sets involving a diameter or a circumradius constraint, it is clear that, up to a translation, the minimizing sequence is contained in a sufficiently big ball. So it remains to study the problem in $\mathcal{K}^2_{P,r}$, $\mathcal{K}^2_{\omega,A}$, $\mathcal{K}^2_{r,\omega}$ and $\mathcal{K}^2_{\omega,P}$. Concerning $\mathcal{K}^2_{P,r}$ and $\mathcal{K}^2_{\omega,P}$, we know from \cite{inequalities_convex} that $P=P(\Omega_k)>2 d(\Om_k)$, for every $k$, so the sequence of the diameters $d(\Om_k)$ is equibounded and, consequently, there exists a sufficiently big ball containing the sequence $(\Om_k)_{k}$.
As far as $\mathcal{K}^2_{r,\omega}$ is concerned, it is possible to prove the boundness of the minimizing sequence whenever $\omega(\Om_k)>2r(\Om_k)$, indeed it holds (see \cite{inequalities_convex})
$$d(\Om_k)\le \frac{\omega^2(\Om_k)}{2(\omega(\Om_k)-2r(\Om_k))}.$$
For the last class $\mathcal{K}^2_{\omega,A}$, from \cite{inequalities_convex}, we know that, if $2\omega(\Om_k)\leq \sqrt{3} d(\Om_k)$, then $$2\abs{\Om_k}\ge \omega(\Om_k)d(\Om_k),$$and, also in this case, the boundedness follows.

So, for every class of sets considered, the Blaschke selection theorem ensures us that, up to a subsequence, $(\Omega_k)$ converges with respect to the Hausdorff distance to a convex set $\Om^*$; it remains only to prove that this set belongs to the relative class of admissible sets. We observe that all the considered constraints are stable for the Hausdorff convergence. In particular, the stability of the inradius is proved in in \cite{delyon2} and the one of the diameter is proved in \cite{pierre}, meanwhile the stability of the area, the perimeter and the width may be found in \cite{schneider}. 

It only remains to show that the circumradius is continuous with respect to the Hausdorff distance in the class of admissible sets having a circumradius constraint. Since $R(\Om_k)=R$, for all $k\in\N$, we have $\Om_k\subseteq B_R$. Using the stability of the Hausdorff convergence for the inclusion (see \cite[Proposition 2.2.17]{pierre}), we have that $\Om^*\subseteq B_R$, and consequently $R(\Om^*)\le R$. By contradiction, let us suppose that $R(\Om^*)<R$, so there exists $\overline{R}>0$ such that $R(\Om^*)<\overline{R}<R$ and so $\Om^*\subseteq B_{\overline{R}}$. Therefore, by the Hausdorff convergence, for sufficiently large $k$, $\Om_k\subseteq B_{\overline{R}}$, but this would imply $R\le \overline{R}$, which is absurd.

In order to conclude, we observe that in all the above cases, the set $\Om^*$ cannot be a segment, that is to say, that the minimizing sequence $(\Om_k)_{k\in \N}$ cannot degenerate loosing one dimension. If we are working in a class of sets involving an inradius or width constraint, then, it is clear that, thanks to the continuity of the inradius and width under the Hausdorff convergence, there exists, up to a translation, a sufficiently small ball contained in the minimizing sequence. 

Moreover, in the case $\mathcal{K}^2_{d,A}$ and $\mathcal{K}^2_{R,A}$, the non-degeneration is ensured by the continuity of the area under Hausdorff distance and the equiboundedness of the diameter. On the other hand, 
if we consider the minimization problem in $\mathcal{K}^2_{R,d}$, $\mathcal{K}^2_{P,d}$ and $\mathcal{K}^2_{P,R}$, the inradius can be bounded from below by a positive quantity. 
In \cite[Section 9]{santalo}, it is proved that 
\begin{equation*}
    r(\Omega_k)\geq \dfrac{d^2(\Omega_k)\sqrt{4 R^2(\Omega_k)-d^2(\Omega_k)}}{2R(\Omega_k)\left(2 R(\Omega_k)+\sqrt{4 R^2(\Omega_k)-d^2(\Omega_k)}\right)}.
\end{equation*}

In \cite[Section 3]{henk}, it is proved that
\begin{equation*}
    r(\Omega_k)\geq \frac{P(\Om_k)}{4}-\frac{d(\Om_k)}{2},
\end{equation*}
which also yields to 
\begin{equation*}
    r(\Omega_k)\geq \dfrac{ P(\Omega_k)}{4}-R(\Omega_k),
\end{equation*}
as $d(\Om_k)\leq 2R(\Om_k)$. 

Recalling now that the Cheeger constant is continuous with respect to the Hausdorff convergence when the sets do not degenerate to a segment (see \cite[Proposition 3.1]{reverse_cheeger}), the existence part of the theorem is proved. 
\end{proof}

\subsection{Proof of Theorem \ref{th2}}
The following paragraphs of this Section are dedicated to the proof of the explicit bounds and their sharpness.

\subsubsection{The triplet $(P,h,r)$}
\begin{proposition}\label{prop_hrP}
Let $\Om\in \mathcal{K}^2$. Then, it holds
\begin{equation}\label{eq:hrP}
\frac{1}{r(\Omega)} +\frac{\pi}{P(\Omega)-\pi r(\Omega)} \leq h(\Omega)\leq \frac{1}{r(\Omega)}+\sqrt{\frac{2\pi }{P(\Omega)r(\Omega)}},    
\end{equation}
where the equality is achieved by sets that are homothetic to their form bodies in the upper bound and by the stadiums in the lower bound. 
\end{proposition}

\begin{proof}
We combine the following classical convex geometric inequalities that can be found in \cite{inequalities_convex}
\begin{equation}\label{pra}
    \frac{P(\Omega)r(\Omega)}{2}\leq |\Omega|\leq r(\Omega)(P(\Omega)-\pi r(\Omega)),
\end{equation}
with the estimates \eqref{eq:hra} to obtain the optimal inequalities \eqref{eq:hrP}.
    
 The upper bound in \eqref{eq:hra} is an equality for sets that are homothetic to their form bodies, since both the upper bound in \eqref{eq:hra} and the lower one in \eqref{pra} are sharp for such sets. 

 The lower bound an equality for stadiums, since both the lower bound in \eqref{eq:hra} and the upper one in \eqref{pra} are sharp for those sets.
\end{proof}


\subsubsection{The triplet $(d,h,r)$}
In the following, we will denote by $\mathcal{S}_{r,d}$ the symmetrical spherical slice of inradius $r$ and diameter $d$ and by $\mathcal{N}_{r,d}$ the regular smoothed nonagon with the same inradius and diameter, see Definitions \ref{def:nonagon} and \ref{def:slice}.
\begin{proposition}\label{prop_hdr} Let $\Omega\in\mathcal{K}^2$. Then, it holds
\begin{equation}\label{up_hdr}
    h(\Omega)\leq \frac{1}{r(\Omega)}+ \sqrt{\frac{\pi}{r(\Om)\sqrt{d^2(\Om)-4r^2(\Om)}+r^2(\Om)\left(\pi-2\arccos{\left(\frac{2r(\Om)}{d(\Om)}\right)}\right)}},
\end{equation}
where the equality is achieved if and only if $\Om$ is a symmetric two-cup body. Moreover, we have
\begin{equation}
    \label{hdr_low}
   h(\Om) \ge  \frac{1}{t_{g_1^\Omega}}, 
\end{equation}
where $t_{g_1^\Om}$ is the smallest solution to the equation
\begin{equation}\label{eq1}
    g_1^\Om(t):=\psi(d(\Om)-2t,r(\Om)-t)=\pi t^2
\end{equation}
 on the interval $[0, r(\Om)]$ and the function $\psi$ is defined in \eqref{func_del}. 
 
 Moreover, there exists $D_0$ such that the problem $$\min\{h(\Om)\ |\ \Om\in\K^2_{d,r}\}$$ is solved by the smoothed nonagon $\mathcal{N}_{r,d}$ if $d<rD_0$ and by the slice $\mathcal{S}_{r,d}$ if $d\ge r D_0$. 
\end{proposition}

\begin{proof} 
To prove \eqref{up_hdr}, one has just to combine the upper bound in \eqref{eq:hra}, which is an equality for sets that are homothetic to their form bodies, and \eqref{del_2cap}, which is an equality for and only for symmetric two-cup bodies, that are also homothetic to their form bodies. 

Let us now prove \eqref{hdr_low}. As in Theorem \ref{delthm} (see \cite{delyon2}), for any constant parameter $r>0$, we consider the function 

\begin{equation*}
   \Psi_r(x) =\begin{cases}
   f_r(x):= \displaystyle{\frac{3\sqrt{3}r}{2}(\sqrt{x^2-3r^2}-r)+\frac{3x^2}{2}\left(\frac{\pi}{3}-\arccos{\left(\frac{\sqrt{3}r}{x}\right)}\right)}, & \text{if} \, \, \,  x\le r D^*\vspace{1mm} \\
    g_r(x):=\displaystyle{r\sqrt{x^2-4r^2}+\frac{x^2}{2}\arcsin{\left(\frac{2r}{x}\right)}}, & \text{if} \, \, \, x\ge r D^*
\end{cases} 
\end{equation*}
defined for $x\ge 2r$. The function $\Psi_r$ is strictly increasing. Indeed, we have for every $x<r D^*$
$$f_r'(x) = \frac{3 \sqrt{3} r x}{2 \sqrt{x^2-3 r^2}} + 3 x \left(\frac{\pi }{3}-\arccos\left(\frac{\sqrt{3} r}{x}\right)\right)-\frac{3 \sqrt{3} r}{2 \sqrt{1-\frac{3 r^2}{x^2}}}= 3 x \left(\frac{\pi }{3}-\arccos\left(\frac{\sqrt{3} r}{x}\right)\right),$$
and for every $x>r D^* $,
$$g_r'(x) = -\frac{r}{\sqrt{1-\frac{4 r^2}{x^2}}}+\frac{r x}{\sqrt{x^2-4 r^2}}+x \arcsin\left(\frac{2 r}{x}\right)=x \arcsin\left(\frac{2 r}{x}\right)>0. $$
Thus, the function $f_r$ is increasing on $[2r,2\sqrt{3}r]$ and is decreasing on $[2\sqrt{3}r,+\infty)$ and the function $g_r$ is increasing on $[2r,+\infty)$. Moreover, we have by Theorem \ref{delthm} that $D^*\leq 2\sqrt{3}$. So, the function $f_r$ is increasing on the sub-interval $[2r,r D^*]$ and, since $f_r(r D^*)=g_r(r D^*)=\Psi_r(D^*)$, the continuous function $\Psi_r$ is increasing on $[2r,+\infty)$. 

Let $t\in [0,r(\Om)]$, by applying the result of Theorem \ref{delthm} on the convex set $\Om_{-t}$, we have
$$|\Om_{-t}|\leq \Psi_{r(\Om_{-t})}(d(\Om_{-t})) = \Psi_{r(\Om)-t}(d(\Om_{-t}))\leq \Psi_{r(\Om)-t}(d(\Om)-2t)=:g_1^\Om(t),$$
where, we use the monotonicity of the function $\Psi_{r(\Om)-t}$ and the estimates \eqref{inr} and \eqref{eq:diameter}.

Now, using Lemma \ref{lem:main}, we have the following bound for the Cheeger constant
\begin{equation}\label{lowerr}
    h(\Om)\ge \frac{1}{t_{g_1^\Om}},
\end{equation}
where $t_{g_1^\Om}$ is the smallest solution to the equation $g_1^\Om(t)=\pi t^2$ on the interval $[0,r(\Om)]$. 

It remains to prove that for every $r>0$ and $d\ge 2r$, there exists a convex set of inradius $r$ and diameter $d$ such that \eqref{lowerr} is an equality. If $d=2r$ then $\Om$ is a ball and thus the equality is trivial. Let us now consider the following two cases: 
\begin{itemize}
    \item If $d\geq  r D^*$, we have, for every $t\in[0,r)$,
    \begin{equation}\label{eq:proof_hdr}
    |(\S_{r,d})_{-t}|=|\S_{r-t,d-2t}|=\Psi_{r-t}(d-2t),    
    \end{equation}
    where the first equality is a consequence of the equality $(\S_{r,d})_{-t}=\S_{r-t,d-2t}$ and the second one is a consequence of \cite[Theorem 2]{delyon2} and of the following estimate
    $$d((\S_{r,d})_{-t})=d-2t>rD^*-2t>rD^*-tD^*=(r-t)D^*=r((\S_{r,d})_{-t})D^*,$$
    where we used $D^*\approx 2,3888>2$ (see \cite[Theorem 2]{delyon2}). Thus, we have by \eqref{eq:proof_hdr} 
    $$ h(\S_{r,d})=\frac{1}{t_{g_1^\Om}},$$
    with 
    $$r(\S_{r,d})= r\ \ \text{and}\ \ d(\S_{r,d})= d.$$
    \item If $d\in (2r,r D^*]$, we consider $t^*:=\frac{D^*r-d}{D^*-2}$, that is the value for which the graphs of the functions $t\longmapsto |(\mathcal{N}_{r,d})_{-t}|$ and $t\longmapsto |(\S_{r,d})_{-t}|$ intersect each other, see Figure \ref{fig:discuss_1}. We note that $(\mathcal{N}_{r,d})_{-t}=\mathcal{N}_{r-t,d-2t}$ for every $t\in [0,t^*]$.

We introduce the quantity  $D_0\in(2,D^*)$ as the (unique) value in the interval $(2,D^*)$ for which the graph of the (decreasing) function $x\longmapsto \frac{D^*-x}{D^*-2}$ intersects the graph of the (increasing\footnote{The function $x\longmapsto \frac{1}{h(\mathcal{N}_{1,x})}$ is increasing because $x\longmapsto \mathcal{N}_{1,x}$ is increasing for the inclusion, meanwhile the Cheeger constant is decreasing with respect to the inclusion.}) one $x\longmapsto \frac{1}{h(\mathcal{N}_{1,x})}$. As shown in Figure \ref{fig:discuss_1}, we have the following cases:\vspace{2mm}
    \begin{itemize}
        \item If $\frac{d}{r}<D_0$, i.e., $t^*>\frac{1}{h(\mathcal{N}_{r,d})}$, we have $h(\mathcal{N}_{r,d})=\frac{1}{t_{g_1^\Om}}$, which means that in this case the smoothed nonagon $\mathcal{N}_{r,d}$ provides the equality in \eqref{hdr_low}.
        
        \item If $\frac{d}{r}>D_0$, i.e., $t^*<\frac{1}{h(\mathcal{N}_{r,d})}$, we have $h(\S_{r,d})=\frac{1}{t_{g_1^\Om}}$, which means that in this case the slice $\mathcal{S}_{r,d}$ provides the equality in \eqref{hdr_low}.
        
        \item If $\frac{d}{r}=D_0$, i.e., $t^*=\frac{1}{h(\mathcal{N}_{r,d})}$, we have $h(\mathcal{N}_{r,d})=h(\S_{r,d})=\frac{1}{t_{g_1^\Om}}$, which means that in this case both the smoothed nonagon $\mathcal{N}_{r,d}$ and the slice $\S_{r,d}$ provide the equality in \eqref{hdr_low}.
    \end{itemize}
    So, the proof is concluded.
\end{itemize}

\begin{center}
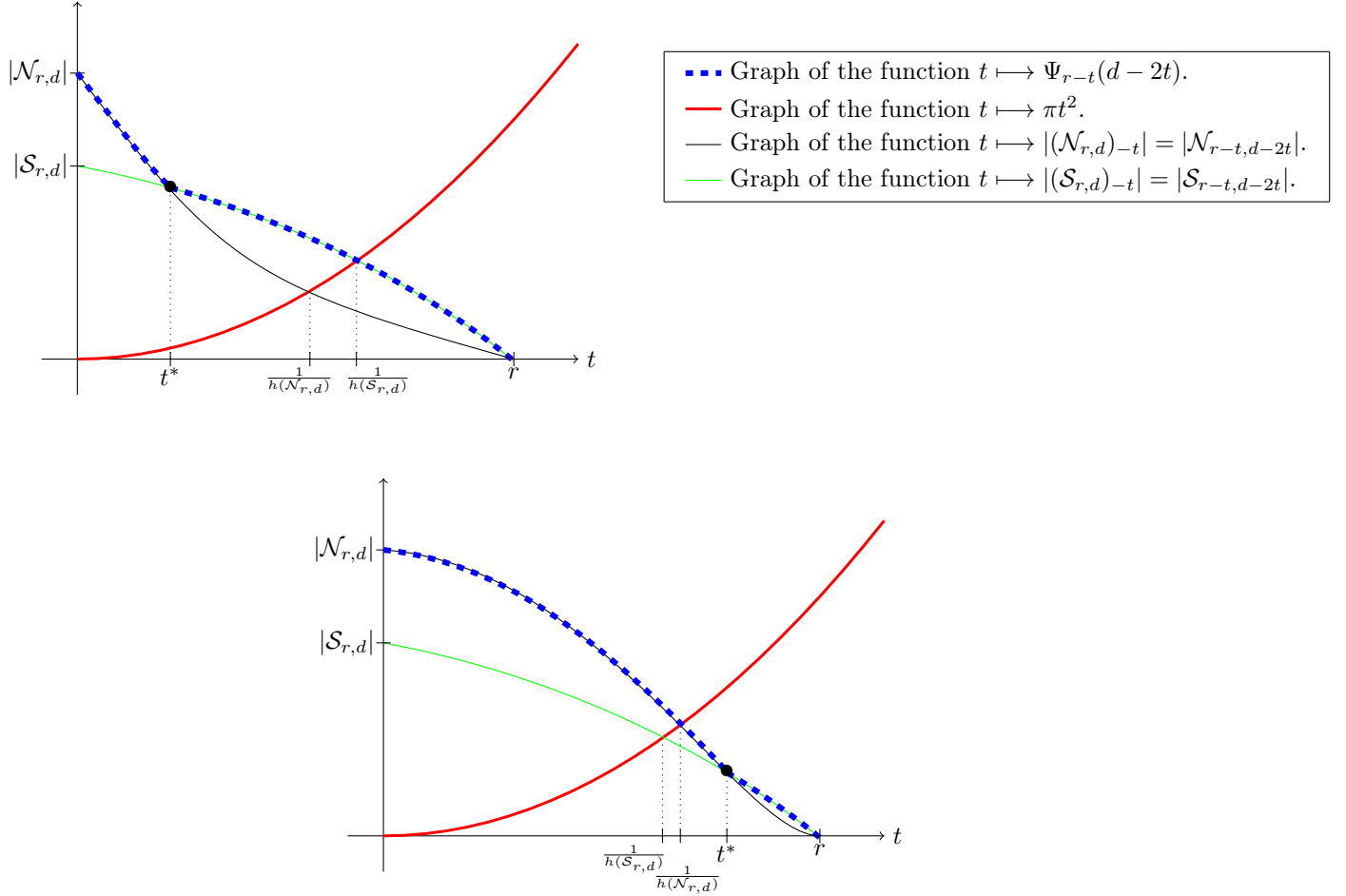
\begin{figure}[h]
\begin{tikzpicture}
\draw[->] (-.5,0) -- (7,0);
\draw (7,0) node[right] {$t$};
\draw [->] (0,-.5) -- (0,5);

\draw [domain=0:7, red,line width=0.4mm] plot(\x,{.09*(\x)^2});

\draw [domain=0:6.1, green] plot(\x,{sqrt(2.7^2-\x)-.03*\x^2});

\draw [domain=1.3:6.1, dashed, blue, line width=0.8mm] plot(\x,{sqrt(2.7^2-\x)-.03*\x^2});

\draw (4.2,0) node[below] {\tiny  $ \frac{1}{h(\mathcal{S}_{r,d})}$};
\draw (3.9,-0.1) -- (3.9,0.1);

\draw (3.1,0) node[below] {\tiny $ \frac{1}{h(\mathcal{N}_{r,d})}$};

\draw (3.25,-0.1) -- (3.25,0.1);
\draw [dotted] (3.25,0) -- (3.25,1);

\draw [dotted] (3.9,0) -- (3.9,1.3);

\draw (1.3,0) node[below] {$t^*$};
\draw (1.3,-0.1) -- (1.3,0.1);

\draw (0,4) node[left] {$|\mathcal{N}_{r,d}|$};
\draw (-0.1,4) -- (0.1,4);

\draw (0,2.7) node[left] {$|\mathcal{S}_{r,d}|$};
\draw (-0.1,2.7) -- (0.1,2.7);

\draw (6.1,0) node[below] {$r$};
\draw (6.1,-0.1) -- (6.1,0.1);

\draw [dotted] (1.3,0) -- (1.3,2.4);

\draw(0,4) .. controls (2,1.3) and (2.5,1)  .. (6.1,0);

\draw[ dashed, blue, line width=0.8mm](0,4) .. controls (1,2.7)  .. (1.3,2.4);

\draw (1.3,2.4) node {\Large $\bullet$};

\draw[ dashed, blue, line width=0.8mm](8.5,4) --  (9,4);
\draw (9,4) node[right] {Graph of the function $t\longmapsto \Psi_{r-t}(d-2t)$.};

\draw[line width=0.4mm, red](8.5,3.5) --  (9,3.5);
\draw(9,3.5) node[right] {Graph of the function $t\longmapsto \pi t^2$.};

\draw (8.5,3) --  (9,3);
\draw(9,3) node[right] {Graph of the function $t\longmapsto |(\mathcal{N}_{r,d})_{-t}|=|\mathcal{N}_{r-t,d-2t}|$.};

\draw [green] (8.5,2.5) --  (9,2.5);
\draw(9,2.5)  node[right] {Graph of the function $t\longmapsto |(\mathcal{S}_{r,d})_{-t}|=|\mathcal{S}_{r-t,d-2t}|$.};

\draw (8.2,2.2) --  (8.2,4.3);
\draw (8.2,2.2) --  (17.5,2.2);
\draw (17.5,2.2) --  (17.5,4.3);
\draw (8.2,4.3) --  (17.5,4.3);
\end{tikzpicture}

\vspace{1cm}
\begin{tikzpicture}
\draw[->] (-.5,0) -- (7,0);
\draw (7,0) node[right] {$t$};
\draw [->] (0,-.5) -- (0,5);

\draw [domain=0:7, red,line width=0.4mm] plot(\x,{.09*(\x)^2});

\draw [domain=0:6.1, green] plot(\x,{sqrt(2.7^2-\x)-.03*\x^2});

\draw [domain=4.8:6.1, dashed, blue, line width=0.8mm] plot(\x,{sqrt(2.7^2-\x)-.03*\x^2});

\draw (3.5,0) node[below] {\tiny  $ \frac{1}{h(\mathcal{S}_{r,d})}$};
\draw (3.9,-0.1) -- (3.9,0.1);

\draw (4.25,-.3) node[below] {\tiny $ \frac{1}{h(\mathcal{N}_{r,d})}$};

\draw (4.15,-0.1) -- (4.15,0.1);
\draw [dotted] (4.15,0) -- (4.15,1.5);

\draw [dotted] (3.9,0) -- (3.9,1.3);

\draw (4.8,0) node[below] {$t^*$};
\draw (4.8,-0.1) -- (4.8,0.1);

\draw (0,4) node[left] {$|\mathcal{N}_{r,d}|$};
\draw (-0.1,4) -- (0.1,4);

\draw (0,2.7) node[left] {$|\mathcal{S}_{r,d}|$};
\draw (-0.1,2.7) -- (0.1,2.7);

\draw (6.1,0) node[below] {$r$};
\draw (6.1,-0.1) -- (6.1,0.1);

\draw [dotted] (4.8,0) -- (4.8,1);

\draw(0,4) .. controls (3,3.8) and (5,0)  .. (6.1,0);

\draw[ dashed, blue, line width=0.8mm](0,4) .. controls (1.6,3.84) and (3,2.84) .. (4.8,.9);

\draw (4.8,.9) node {\Large $\bullet$};
\end{tikzpicture}
\caption{Different cases of equality in inequality \eqref{hdr_low}.}
\label{fig:discuss_1}
\end{figure}
\end{center}

\end{proof}

\begin{remark}
We note that the symmetrical slices and the smoothed nonagons are not the only sets solving the shape optimization problem $\min\{h(\Om)\ |\ \Om\in \K^2_{d,r}\}$. Indeed, if for example we consider a spherical slice $\mathcal{S}$ and denote by $C_\S$ its Cheeger set, we have $h(\S)=h(C_\S)$ and, by the explicit characterization of the Cheeger sets given in \cite[Theorem 1]{kawohl}, we have
 $$r(C_\S)=r\left(\S_{-\frac{1}{h(\S)}}+\frac{1}{h(\S)}B_1\right)=r\left(\S_{-\frac{1}{h(\S)}}\right)+\frac{1}{h(\S)}= r(\S)-\frac{1}{h(\S)}+\frac{1}{h(\S)}=r(S)$$
and 
$$d(C_\S)=d(\S),$$
meanwhile $\S\ne C_\S$, which proves the non-uniqueness of the solution of the minimization problem $$\min\{h(\Om)\ |\ \Om\in \K^2_{d,r}\}.$$ 
\end{remark}

\begin{remark}\label{rk:hdr}
We give the following explicit lower bound. 
In \cite{inequalities_convex}, it is proved that 
\begin{equation*}\label{scotty}
   |\Om|< 2 d(\Om)r(\Om). 
\end{equation*}
By applying the strategy of Lemma \ref{lem:main}, we obtain that 
\begin{equation}\label{eq:remark44}
    h(\Omega) \geq \dfrac{4-\pi}{d(\Om)+2 r(\Om)-\sqrt{(d(\Omega)+2r(\Om))^2-2(4-\pi)d(\Om)r(\Om)}}.
\end{equation}
This estimate is asymptotically achieved by spherical slices with increasing diameter, as shown in Figure 
\ref{fig:remark44}.
\begin{figure}[h]
    \centering
    \includegraphics[scale=.6]{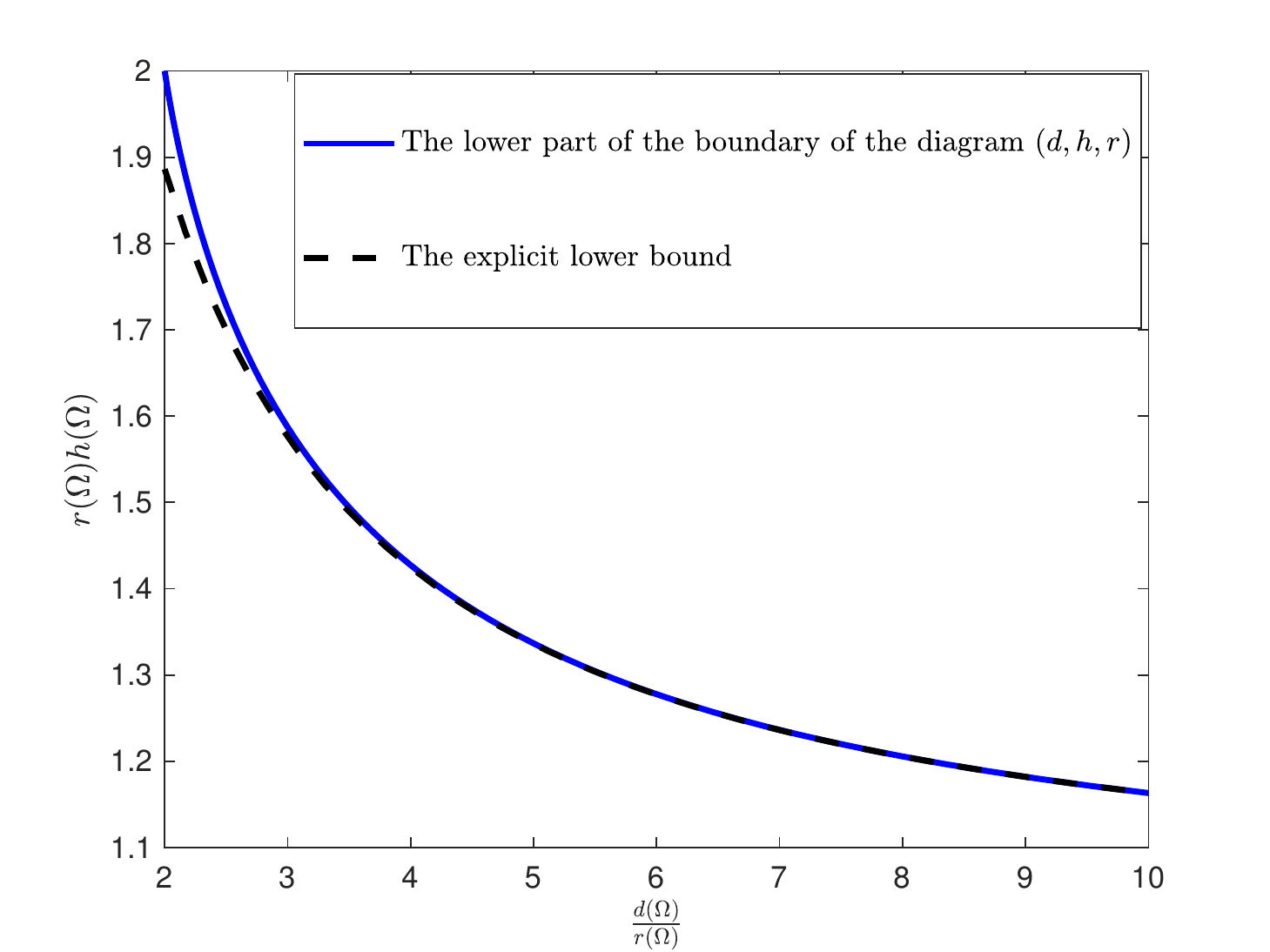}
    \caption{The explicit inequality \eqref{eq:remark44} and the lower part of the boundary of the diagram $(d,h,r)$.}
    \label{fig:remark44}
\end{figure}
\end{remark}

\subsubsection{The triplet $(R,h,r)$ }
\begin{proposition}\label{prop_hRr}
Let $\Omega\in\mathcal{K}^2$. We have 
\begin{equation}\label{hrRlow}
    h(\Om) \ge \frac{1}{t_{g_2^\Omega}},
\end{equation}
where $t_{g_2^\Om}$ is the smallest solution of the equation
\begin{equation}\label{eq2}
    {g_2^\Om}(t):=2\left(\left(r-t\right)\sqrt{(R(\Omega)-t)^2-(r(\Om)-t)^2}+(R(\Om)-t)^2\arcsin{\left(\frac{r(\Om)-t}{R(\Om)-t}\right)}\right)=\pi t^2.
\end{equation}
The equality in \eqref{hrRlow} is achieved if and only if $\Omega$ is a symmetrical spherical slice. Moreover, 
\begin{equation}
    \label{hrRup}
    h(\Omega)\le \frac{1}{r(\Om)}+ \sqrt{\frac{\pi }{2r(\Om)\left(\sqrt{R(\Om)^2-r(\Om)^2}+r(\Om)\arcsin\left(\frac{r(\Om)}{R(\Om)}\right)\right)}},
\end{equation}
where the equality in \eqref{hrRup} is achieved by two-cup bodies.
\end{proposition}

\begin{proof}
In order to prove \eqref{hrRlow}, we apply the result of Lemma \ref{lem:main}. Let us introduce the function $$\varphi:(R,r)\longmapsto 2\left(r\sqrt{R^2-r^2}+R^2\arcsin{\frac{r}{R}}\right),$$ which is increasing with respect to the first variable, indeed
$$\frac{\partial\varphi}{\partial R}(R,r)=2R \arcsin\left(\frac{r}{R}\right)>0.$$
By applying \eqref{cifreRr} (where the equality holds only  for symmetrical spherical slices), we have, for every $t\in [0,r(\Om)]$,
$$|\Om_{-t}|\leq \varphi(R(\Om_{-t}),r(\Om_{-t}))=\varphi(R(\Om_{-t}),r(\Om)-t)\leq \varphi(R(\Om)-t,r(\Om)-t)=:{g_2^\Om}(t),$$
where the last inequality is a consequence of the monotonicity of the function $R\longmapsto \varphi(R,r)$ and of  the fact that $R(\Om_{-t})\leq R(\Om)-t$ (see Lemma \ref{lem:diameter_inner_set}). Finally, we conclude by applying the result of Lemma \ref{lem:main}.

In order to prove \eqref{hrRup}, we combine the upper bound in \eqref{eq:hra} 
and inequality \eqref{secondhRr}.
As far as the sharpness of \eqref{hrRup} is concerned, we observe that \eqref{eq:hra} is sharp on sets that are homothetic to their form bodies and \eqref{secondhRr} is attained by symmetric two-cup bodies, that are also sets that are homothetic to their form bodies. 
\end{proof}

\begin{remark}\label{rk:hrR}
We have the following explicit lower bound
\begin{equation}\label{eq:remark46}
h(\Om)\ge \frac{4-\pi}{2(R(\Om)+r(\Om))-\sqrt{4(R(\Om)+r(\Om))^2-4(4-\pi)R(\Om)r(\Om)}}.
\end{equation}
The inequality can be obtained by combining
$$\abs{\Om}\le 4R(\Om)r(\Om),$$
see \cite{henk}, and the strategy from Lemma \ref{lem:main}. Moreover, the inequality is asymptotically achieved by spherical slices with increasing circumradius, see Figure \ref{fig:remark46}.
\begin{figure}[h]
    \centering
    \includegraphics[scale=.6]{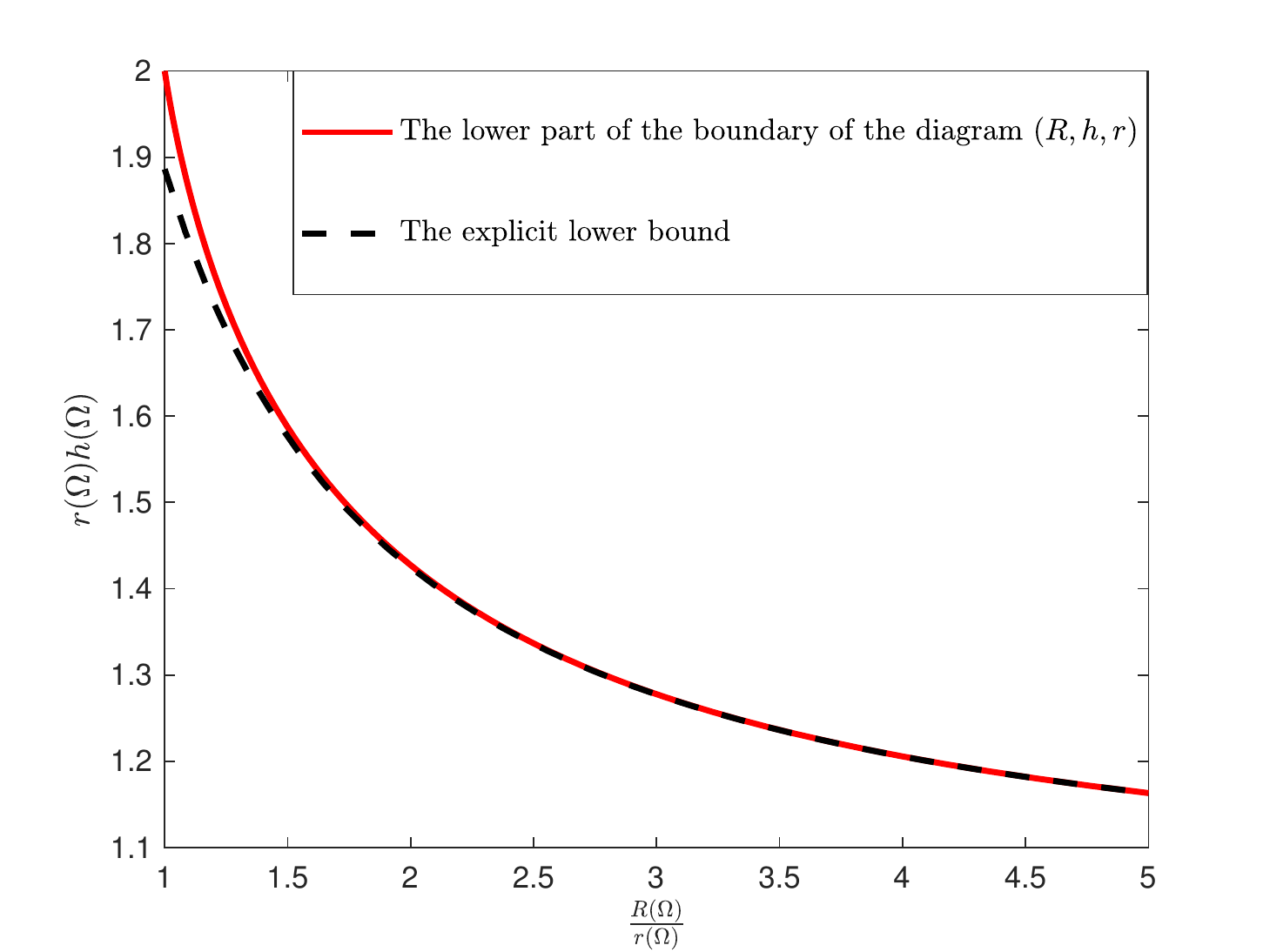}
    \caption{The explicit inequality \eqref{eq:remark46} and the lower part of the boundary of the diagram $(R,h,r)$.}
    \label{fig:remark46}
\end{figure}
\end{remark}
\subsection{Explicit description of the Blaschke-Santal\'o diagrams}
We denote by $\mathcal{D}_1$, $\mathcal{D}_2$ and $\mathcal{D}_3$ the Blaschke Santal\'o diagram respectively corresponding to the triplets $(P,h,r)$, $(d,h,r)$ and $(R,h,r)$. 
We have defined the quantities $t_{g_1^\Om}$ and $t_{g_2^\Om}$ respectively as the smallest solution on $[0,r(\Om)]$ of equations \eqref{eq1} and \eqref{eq2}. We observe that $t_{g_1^\Om}$ depends on $r(\Om)$ and $d(\Om)$, and $t_{g_2^\Om}$ depends on $r(\Om)$ and $R(\Om)$. 

{In the following proposition, we are keeping the inradius $r(\Omega)=1$ and consider different values of the remaining variables: diameter for the diagram $\D_2$ and circumradius for $\D_3$. We then use the notation:} 
$$ t_{g_1^\Om}=t_{g_1}(x) \ \ \text{when $d(\Om)=x$\ \  and}  \quad  t_{g_2^\Om}=t_{g_2}(x)\ \  \text{when $R(\Om)=x$}. $$


\begin{proposition}\label{blaschke}
We obtain the following description of the Blaschke-Santal\'o diagrams

\smaller
\begin{equation}\label{bla1}
    \mathcal{D}_1:=\left\lbrace  \left( P(\Om), h(\Om)\right)\ |\  \Om\in \mathcal{K}^2, \, r(\Om)=1\right\rbrace =\left\lbrace  (x, y)\ |\ x\ge 2\pi\ \ \text{and}\ \  1+ \frac{\pi}{x-\pi}\le y\le 1+\sqrt{\frac{2\pi}{x}}\right\rbrace.
\end{equation}

\begin{equation}\label{bla2}
    \mathcal{D}_2:=\left\lbrace  \left( {R(\Om)}, h(\Om)\right)\ |\ \Om\in \mathcal{K}^2, \, r(\Om)=1\right\rbrace =\left\lbrace  (x, y)\ |\ x\ge 1\ \ \text{and}\ \  \frac{1}{t_{g_2}(x)}\le y\le 1+\sqrt{\frac{\pi}{2\left(\sqrt{x^2-1}+ \arcsin{\frac{1}{x}}\right)}}\right\rbrace.
\end{equation}

\begin{equation}\label{bla3}
    \mathcal{D}_3:=\left\lbrace  \left({d(\Om)}, h(\Om)\right)\ |\  \Om\in \mathcal{K}^2, \, r(\Om)=1\right\rbrace =\left\lbrace  (x, y)\ |\  x\ge 2\  \text{and}\    \frac{1}{t_{g_1}(x)}\le y\le 1+\sqrt{\frac{\pi}{\sqrt{x^2-4}+\left(\pi-2\arccos{\frac{2}{x}}\right)}}\right\rbrace.
\end{equation}
\normalsize

\end{proposition}

\begin{proof}
\begin{itemize}
\item \underline{The diagram $\D_1$ of the triplet $(P,h,r)$:}\\ \vspace{2mm}
   { Let us prove that \eqref{bla1} holds. The bounds proved in Proposition \ref{prop_hrP} ensure that 
$$D_1\subseteq \left\lbrace  (x, y)\ |\ x\ge 2\pi\ \ \text{and}\ \ 1+ \frac{\pi}{x-\pi}\le y\le 1+\sqrt{\frac{2\pi}{x}}\right\rbrace=:\mathcal{P}.$$
It remains to prove the converse inclusion. First, we observe that the boundary of $\mathcal{P}$ is included in $\mathcal{D}_1$. Indeed, we can explicitly construct a family of convex sets which fill the lower part of the boundary of $\mathcal{P}$: let us consider the family of stadiums $\Set{R_l}_{l\ge 0}$ obtained as the convex hull of two balls of radius $1$ and centered in $(0,-l/2)$ and $(0,l/2)$. We have
$$P(R_l)=2\pi+ 2l, \quad r(R_l)=1\ \ \text{and}\  \quad h(R_l)=\frac{P(R_l)}{\abs{R_l}}=\frac{2\pi+ 2l}{\pi+ 2l}=1+ \frac{\pi}{P(R_l)-\pi}.$$
Let us now construct a family of convex sets filling the upper part of the boundary of $\mathcal{P}$. We consider the family of two-cup bodies $\Set{C_k}_{k\ge1}$ obtained as the convex hull of the ball $B_1$ centered at the origin and the points $(-k,0)$ and $(k,0)$. By \cite{cifre_salinas}, we have
$$P(C_k)= 2\left(\sqrt{4k^2-4}+2 \arcsin{\frac{1}{k}}\right), \quad r(C_k)=1 \ \ \text{and}\ \  \abs{C_k}=\left(\sqrt{4k^2-4}+2 \arcsin{\frac{1}{k}}\right)$$
and, as it is shown in \cite{ftJMAA}, 
$$h(C_k)=\frac{1}{r(\C_k)}+\sqrt{\frac{\pi}{\abs{C_k}}}=1+ \sqrt{\frac{2\pi}{P(C_k)}}.$$}

{In order to conclude, we show that the diagram $\D_1$ is vertically convex, i.e., that we can always construct continuous and vertical paths included in the diagram and connecting the upper and the lower parts of the boundary of $\D_1$, covering, in this way, all the area between the two curves. }

{Let $x_0\ge 2\pi$.  There exist $R_{l_0}\in \Set{R_l} $ and $C_{k_0}\in \Set{C_k}$ such that $P(R_{l_0})=P(C_{k_0})=x_0$. 
Let us define, via the Minkowski sum, the set $$K_t=t R_{l_0}+ (1-t) C_{k_0}.$$}
{By the linearity of the perimeter with respect to the Minkowski sum, we have 
$$\forall t\in[0,1],\ \ \ P(K_t)= P(t R_{l_0}+ (1-t) C_{k_0}) = t P(R_{l_0}) + (1-t) P(C_{k_0}) = t x_0 + (1-t)x_0 = x_0.$$ 
As for the inradius, let us consider the unit ball $B_1$ centred at the origin and the rectangle $Q$ of vertices $(-M,-1)$, $(-M,1)$, $(M,1)$ and $(M,-1)$, where $M>0$ is sufficiently large such that both $R_{l_0}$ and $C_{k_0}$ are contained in $Q$. We then have for any $t\in[0,1]$,
$$B_1\subset t R_{l_0}+ (1-t) C_{k_0} \subset Q,$$
which yields by the monotonicity of the inradius with respect to inclusions to
$$\forall t\in[0,1],\ \ \ 1 = r(B_1) \leq r(t R_{l_0}+ (1-t) C_{k_0}) \leq r(Q) = 1.$$
Thus, $$\forall t\in[0,1],\ \ \ r(K_t) = 1.$$
On the other hand, by the continuity of the Cheeger constant with respect to the Hausdorff distance, we have by the intermediate value theorem
$$[h(R_{t_0}),h(C_{k_0})] \subset \{h(K_t)\ |\ t\in[0,1]\}.$$
Thus, since $\{h(K_t)\ |\ t\in[0,1]\}\subset [h(R_{t_0}),h(C_{k_0})]$ (because $R_{t_0}$ and $C_{k_0}$ respectively 
 correspond to points laying on the lower and the upper parts of the boundary of the diagram $\D_1$), we have the equality $$[h(R_{t_0}),h(C_{k_0})] = \{h(K_t)\ |\ t\in[0,1]\}.$$
 This proves that the diagram $\D_1$ is vertically convex and that $\D_1 = \mathcal{P}$.
}
\vspace{2mm}

\item \underline{The diagram $\D_2$ of the triplet $(R,h,r)$:} \vspace{2mm}

{The proof of \eqref{bla2} follows the same scheme. Indeed, one can prove that the diagram is vertically convex by considering vertical paths constructed via the Minkowski sums of the extremal sets (those corresponding to points laying on the upper and lower parts of the boundary of the diagram). In the present case, by Proposition \ref{prop_hRr}, the upper boundary is filled by points corresponding to two-cup bodies $(C_x)_{x\ge 1}$ such that $R(C_x)=x$ and $r(C_x) = 1$, for all $x\ge 1$. Meanwhile, the lower boundary is filled by points corresponding to spherical slices $(S_x)_{x\ge 1}$ such that $R(S_x)=x$ and $r(S_x) = 1$, for all $x\ge 1$. If one assumes (without loss of generality) that for all $x\ge 1$, both $C_x$ and $S_x$ are centered at the origin and that their diameters are colinear, then, it is clear that $C_x\subset S_x$. Thus, it is straightforward that $C_x\subset (1-t)C_x + t S_x \subset S_x$, for all $t\in[0,1]$. Therefore, we have by the monotonicity of the circumradius and the inradius with respect to inclusions that
$$\forall t\in[0,1],\ \ \ 1= r(C_x)\leq r((1-t)C_x + t S_x) \leq r(S_x) = 1$$
and 
$$\forall t\in[0,1],\ \ \ x= R(C_x)\leq R((1-t)C_x + t S_x) \leq R(S_x) = x.$$
Thus, 
$$\forall t\in[0,1],\ \ \ r((1-t)C_x + t S_x)=1\ \ \text{and}\ \  R((1-t)C_x + t S_x) = x,$$
which allows to conclude the vertical convexity of the diagram as in the previous case.} \vspace{2mm}



\item \underline{The diagram $\D_3$ of the triplet $(d,h,r)$:} \vspace{2mm}

{As for the proof of \eqref{bla3}, let $x_0\ge 2$ and $C_{x_0}$ be a two-cup body such that $r(C_{x_0})=1$ and $d(C_{x_0})=x_0$ (corresponding to a point on the upper boundary of $\D_3$, see Proposition \ref{prop_hdr}) and $L_{x_0}$ be an extremal shape corresponding to a point on the lower boundary such that $r(C_{x_0})=1$ and $d(C_{x_0})=x_0$. In this case, we have to be more careful as we should distinguish two cases: 
\begin{itemize}
    \item If $x_0 \ge D_0$, then by Proposition \ref{prop_hdr}, $L_{x_0}$ is a symmetrical spherical slice whose diameter can be assumed to be colinear to the diameter of $C_{x_0}$. In this case, we can conclude exactly as in the previous case by considering the convex Minkowski combination of $C_{x_0}$ and $L_{x_0}$.
    \item If $x_0\in [2,D_0)$, then $L_{x_0}$ can be taken as a smoothed nonagon of inradius $1$ and diameter $x_0$ (see Proposition \ref{prop_hdr}). In this case, the Minkowski sum does no longer allow to construct a vertical line included in the diagram and joining the upper and the lower parts of its boundary. Therefore, to prove the vertical convexity, we use the following  construction introduced in \cite[Section 1.2]{delyon2}:  Starting from the smoothed nonagon $L_{x_0}$, we fix one of its diameter, say $[A, B]$ and we shrink continuously $L_{x_0}$ to the set $L_{AB}$ defined as the convex hull of the points $A$, $B$ and the disk of radius $1$
    contained in $L_{x_0}$. Then, we continuously move the points $A$ and $B$  to the points $A_0$
 and  $B_0$ at distance $x_0$,
oppositely located with respect to the center of the disk (in the sense that the center is the middle of
$[A_0,B_0]$) by keeping the convex hull with the disk through the perturbation process, obtaining the two-cup body $C_{x_0}$ as a final shape. By doing so, we constructed a continuous family of convex shapes (of inradius $1$ and diameter $x_0$)  connecting $L_{x_0}$ and $C_{x_0}$, yielding a vertical and continuous line connecting the upper and the lower parts of the boundary of the Blaschke--Santal\'o diagram $\D_3$. 
\end{itemize}}


\end{itemize}

\end{proof}

\begin{remark}
As observed in this section, proving  sharp bounds corresponding to the boundary of the Blaschke--Santal\'o diagram is not equivalent to completely characterizing it. Indeed, once we managed to identify the boundary, it remains to show that the diagram is simply connected, which can be a difficult task, see for example \cite[Conjecture 5]{AH11}, \cite[Open problem 2]{LZ} and \cite[Problem 3]{vdBBP}.  We refer to \cite{FL21} for a generic method of proof of the simple connectedness of a Blaschke--Santal\'o diagram. 
\end{remark}

\section{The remaining diagrams}\label{partial}
For the remaining triplets of shape functionals, we have proved the existence of a maximum and minimum to the relative shape optimization problems (see Theorem \ref{th:existence}). Moreover, we are able to identify parts of the boundaries of the corresponding Blaschke--Santal\'o diagrams. In the following, we state and prove the results that we have obtained. 
\subsection{The triplet $(\omega,h,d)$}
\begin{proposition}\label{prop_hdw}
Let $\Omega\in\mathcal{K}^2$. We have 
\begin{equation}\label{hdwlow}
h(\Om)\ge \frac{1}{t_{g_3^\Om}},
\end{equation}
where $t_{g_3^\Om}$ is the smallest solution to
$${g_3^\Om}(t):=\frac{\omega(\Om)-2t}{2}\sqrt{(d(\Om)-2t)^2-(\omega(\Om)-2t)^2}+ \frac{(d(\Om)-2t)^2}{2}\arcsin\left(\frac{\omega(\Om)-2t}{d(\Om)-2t}\right)=\pi t^2.$$
The equality in \eqref{hdwlow} is achieved by symmetrical spherical slices. 
Moreover,
\begin{itemize}
    \item if $\omega(\Om)\le \frac{\sqrt{3}}{2}d(\Om)$, we have
\begin{equation}\label{hdwup}
   h(\Om)\le  h(T_I),
\end{equation}
where $T_I$ is the subequilateral triangle such that $\omega(T_I)=\omega(\Om)$ and $d(T_I)=d(\Om)$. The equality in \eqref{hdwup} is achieved by the isosceles triangle $T_I$; 
\item and if $\frac{\sqrt{3}}{2}d(\Om)\leq\omega(\Om) \leq d(\Om)$, we have
\begin{equation}\label{hdwup_1}
h(\Om)\le \frac{\sqrt{3}}{\sqrt{3}\omega(\Om)-d(\Om)}+\sqrt{\frac{2\pi}{\pi \omega(\Om)^2-\sqrt{3}d(\Om)^2+6\omega(\Om)(\tan\left(\arccos(\frac{\omega(\Om)}{d(\Om)})\right)-\arccos(\frac{\omega(\Om)}{d(\Om)}))}}. 
\end{equation}
The equality in \eqref{hdwup_1} is achieved by equilateral triangles. 
\end{itemize}

\end{proposition}

\begin{proof}
\begin{itemize}
    \item Let us start by proving the lower bound \eqref{hdwlow}, by using the strategy given in Lemma \ref{lem:main}. 
For every $\Om\in\mathcal{K}^2$, it holds
\begin{equation*}
    \label{Adw_up} 
    \abs{\Om}\le \frac{\omega(\Om)}{2}\sqrt{d^2(\Om)-\omega^2(\Om)}+ \frac{d^2(\Om)}{2}\arcsin\left(\frac{\omega(\Om)}{d(\Om)}\right),
\end{equation*}
see \cite{kubota,inequalities_convex}, with equality if and only if $\Om$ is a symmetrical spherical slice. If we denote  
$$f(d,w):=\frac{w}{2}\sqrt{d^2-w^2}+\frac{d^2}{2}\arcsin\left(\frac{w}{d}\right),$$
we have
$$\frac{\partial f}{\partial d}(d,w)= d\arcsin\left(\frac{w}{d}\right)>0\ \ \ \text{and}\ \ \ \frac{\partial f}{\partial w}(d,w)=\sqrt{d^2-w^2}>0.$$
Thus, 
$$\abs{\Om_{-t}}\le f(d(\Om_{-t}), \omega(\Om_{-t}))\le f(d(\Om)-2t, \omega(\Om)-2t).$$
Therefore, by using Lemma \ref{lem:main}, we have
$$h(\Om)\ge \frac{1}{t_{g_3^\Om}},$$
where $t_{g_3^\Om}$ is the smallest solution of
$${g_3^\Om}(t):=f(d(\Om)-2t, \omega(\Om)-2t)=\pi t^2.$$

\item In order to prove the upper bound \eqref{hdwup}, we consider the following minimization problem for the area in the class of convex sets with given diameter and width, studied in  \cite{inequalities_convex} and \cite{sholander}: 
\begin{enumerate}[label=(\roman*)]
\item if $2\omega(\Om)\le \sqrt{3}d(\Om)$, then 
\begin{equation}\label{min1}
   2\abs{\Om}\ge \omega(\Om)d(\Om), 
\end{equation}
where the equality is achieved by triangles;
\item if $\sqrt{3}d(\Om)\le 2\omega(\Om)\le 2 d(\Om)$, then 
\begin{equation}\label{min2}
    \displaystyle{2\abs{\Om}\ge\pi \omega^2(\Om)-\sqrt{3}d^2(\Om)+6\omega^2(\Om)\left(\tan\left(\arccos\left(\frac{\omega(\Om)}{d(\Om)}\right)\right)-\arccos\left(\frac{\omega(\Om)}{d(\Om)}\right)\right)}=|T_Y|, 
\end{equation}
where the equality is achieved by the Yamanouti set $T_Y$ such that $\omega(T_Y)=\omega(\Om)$ and $d(T_Y)=d(\Om)$.
\end{enumerate}

 Moreover, if we consider the minimization problem of the inradius in the class of convex sets with given diameter and width, we have from Theorem \ref{cifre_wrd_thm}  
 \begin{equation}\label{min3}
   r(\Om)\geq  \begin{cases}
 r(T_I), & \text{if}\ \, 2\omega(\Om)\le \sqrt{3}d(\Om),\\
 r(T_Y), & \text{if}\ \, \sqrt{3}d(\Om)\le 2\omega(\Om)\le 2 d(\Om).
 \end{cases}
  \end{equation}
So combining \eqref{eq:hra} with the estimates \eqref{min1}, \eqref{min2} and \eqref{min3}, we obtain 

 \begin{equation*}
   h(\Om)\leq  \begin{cases}
\frac{1}{r(T_I)}+\sqrt{\frac{\pi}{|T_I|}}= h(T_I), & \text{if}\ \, 2\omega(\Om)\le \sqrt{3}d(\Om),\\
 \frac{1}{r(T_Y)}+\sqrt{\frac{\pi}{|T_Y|}}, & \text{if}\ \, \sqrt{3}d(\Om)\le 2\omega(\Om)\le 2 d(\Om).
 \end{cases}
  \end{equation*}
The explicit formula given in inequality \eqref{hdwup_1} is then obtained by using \eqref{min2} and $r(T_Y)=\omega(T_Y)-\frac{d(T_Y)}{\sqrt{3}}$, see \cite[Theorem 2]{cifre4}. 
\end{itemize}
\end{proof}

\begin{figure}[h]
    \centering
    \includegraphics[scale=.26]{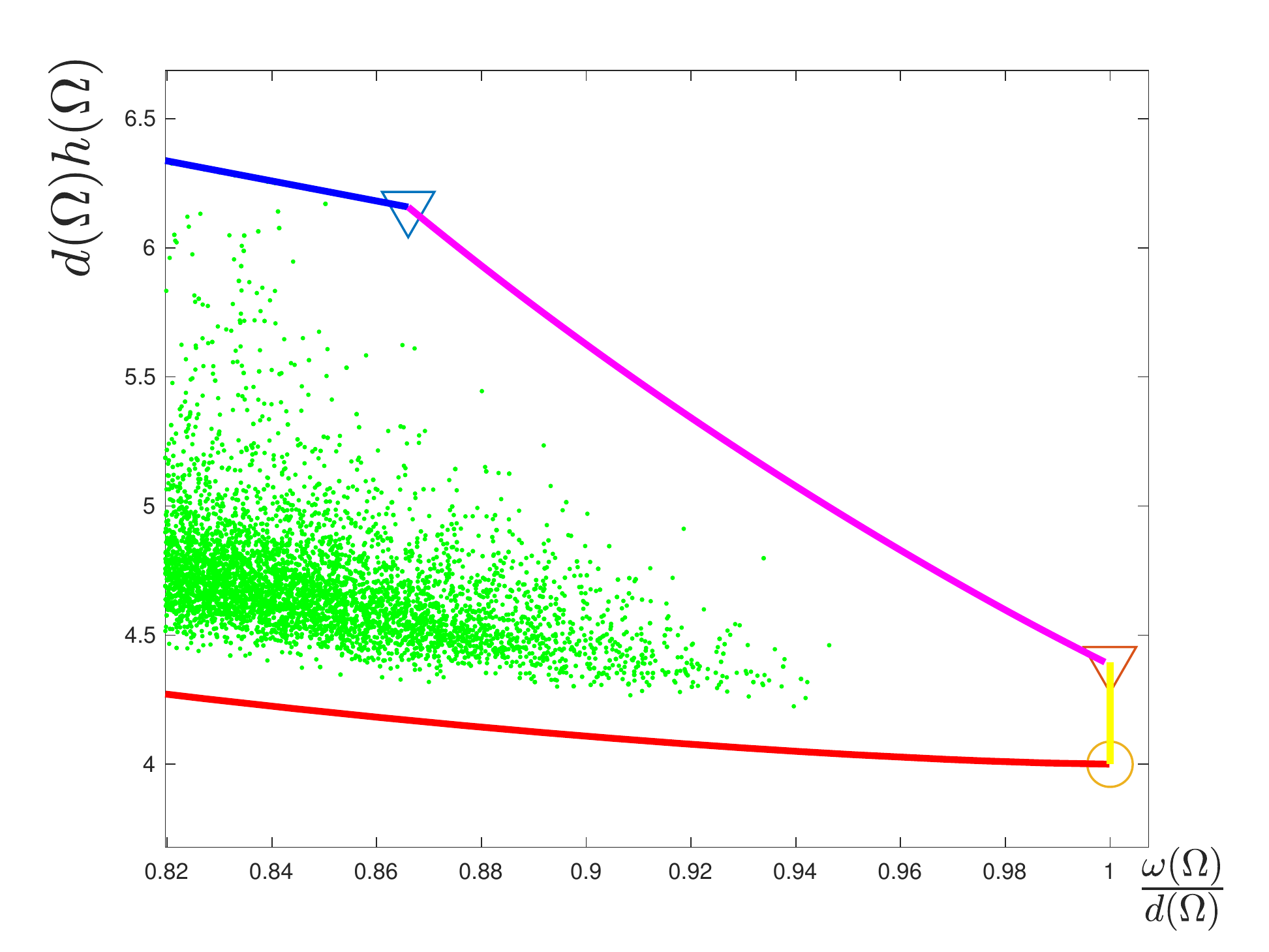}
    \includegraphics[scale=.26]{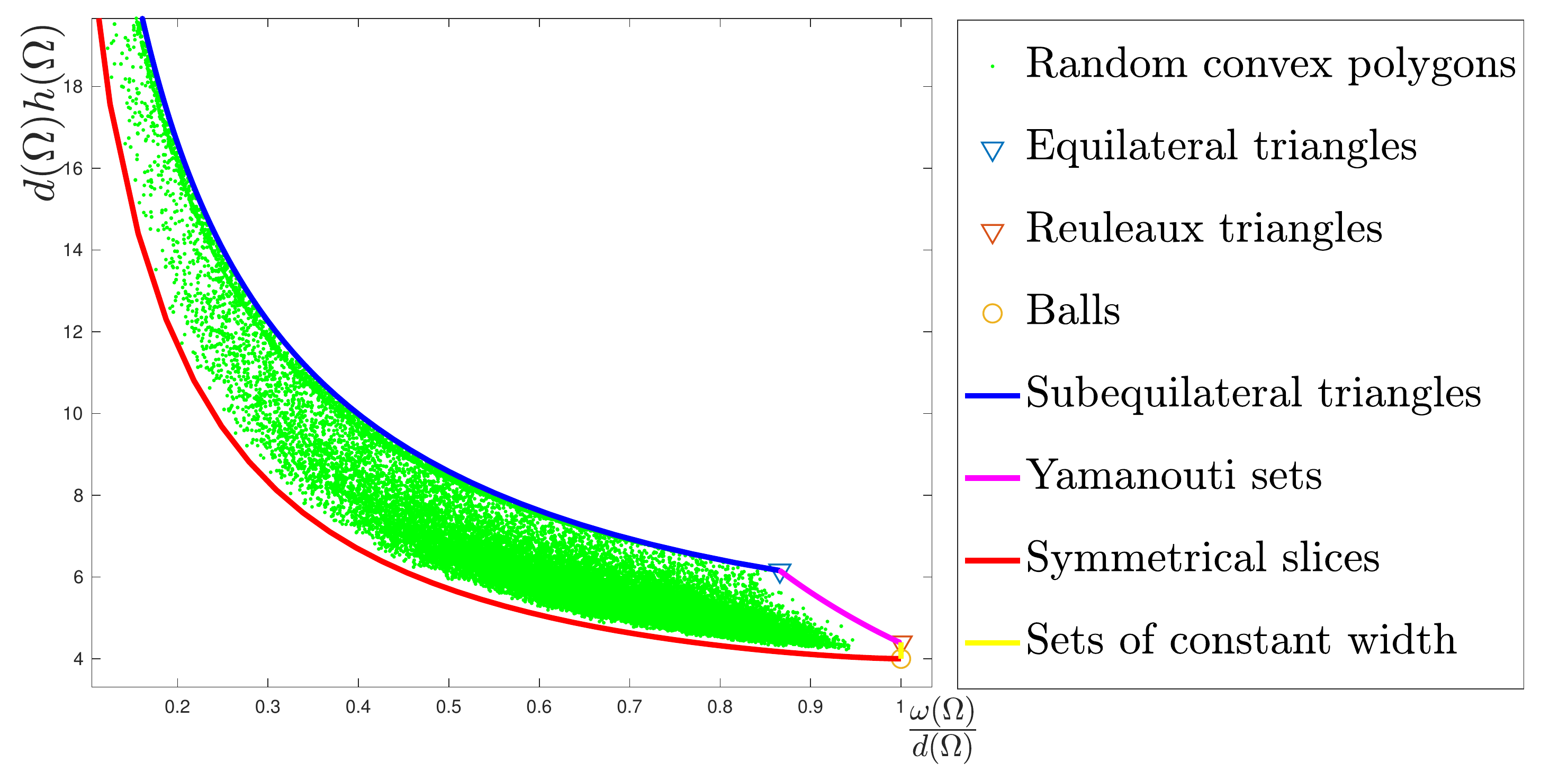}
    \caption{Blaschke--Santal\'o diagram of the triplet $(\omega,h,d)$.}
    \label{fig:hwd}
\end{figure}
\newpage 

We are also able to give an explicit sharp lower bound for the Cheeger constant in terms of the width and the diameter. We note that although being sharp, this inequality does not correspond to a part of the boundary of the Blaschke--Santal\'o diagram as shown in Figure \ref{fig:remark52}. 

\begin{remark}\label{rk:hwd}
Let $\Om\in \K^2$. We have
\begin{equation}\label{lr}
    h(\Om)>\frac{1}{\omega(\Om)}+\frac{1}{d(\Om)}+\sqrt{\left(\frac{1}{\omega(\Om)}+\frac{1}{d(\Om)}\right)^2-\frac{4-\pi}{\omega(\Omega)d(\Om)}},
\end{equation}
where the equality is asymptotically achieved by a sequence of thin collapsing rectangles or spherical slices as shown in Figure 
\ref{fig:remark52}.

In order to prove \eqref{lr}, it is enough to consider the inequality
$$\abs{\Om}\le \omega(\Om) d(\Om)$$
and to use the strategy of Lemma \ref{lem:main}.

\begin{figure}[h]
    \centering
    \includegraphics[scale=.6]{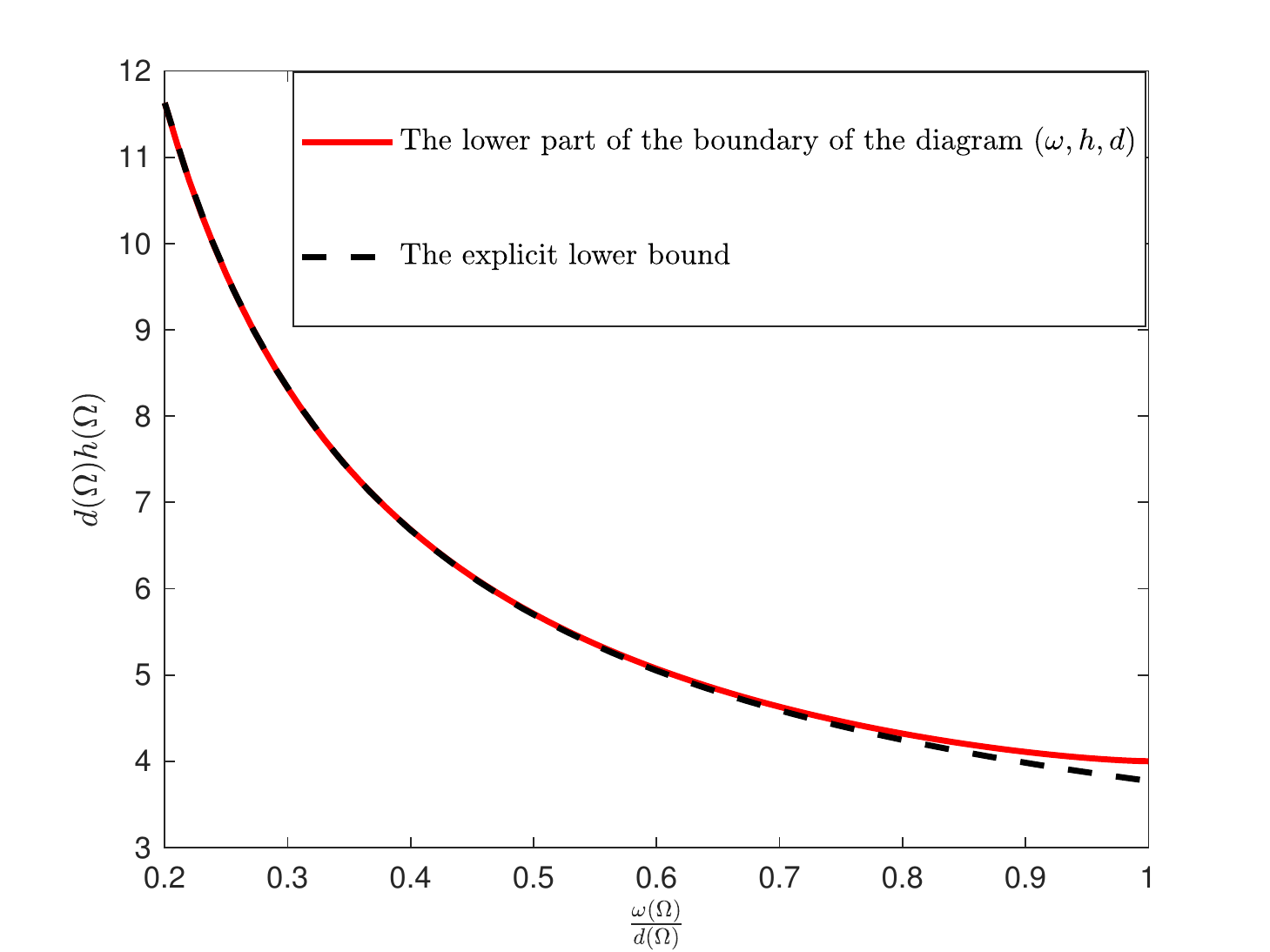}
    \caption{The explicit inequality \eqref{lr} and the lower part of the boundary of the diagram $(\omega,h,d)$.}
    \label{fig:remark52}
\end{figure}
\end{remark}

\subsection{ The triplet $(\omega, h, R)$}
\begin{proposition}\label{prop_hRw}
Let $\Omega\in\mathcal{K}^2$. We have 
\begin{equation}\label{low_hRw}
  h(\Om)\ge \frac{1}{t_{g_4^\Om}},
\end{equation}
where $t_{g_4^\Om}$ is the smallest solution of the equation
$$g_4^\Om(t):=  \frac{(\omega(\Om)-2t)}{2} \sqrt{4 \left(R(\Om)-t\right)^2-\left(\omega(\Om)-2t\right)^2} + 2(R(\Om)-t)^2 \arcsin{\left(\frac{\omega(\Om)-2t}{2(R(\Om)-t)}\right)}= \pi t^2$$
on $[0,r(\Om)]$. The equality in \eqref{low_hRw} is achieved by symmetrical spherical slices.

Moreover, if $\omega(\Om)\le \frac{3}{2}R(\Omega)$, then
\begin{equation}\label{upper_hRw}
h(\Omega)\leq h(T_I), 
\end{equation}
where $T_I$ is the subequilateral triangle such that $R(T_I)=R(\Om)$ and $\omega(T_I)=\omega(\Om)$. The equality in \eqref{upper_hRw} is achieved by the subequilateral triangle $T_I$. 
\end{proposition}

\begin{proof}
\begin{itemize}
    \item Let us start by proving the lower bound \eqref{low_hRw}, by using the strategy given in Lemma \ref{lem:main}. 
Let us recall the function defined in \eqref{henkchi} $$\chi:(\omega, R)\longmapsto \frac{\omega}{2} \sqrt{4 R^2-\omega^2} + 2R^2 \arcsin{\frac{\omega}{2R}}.$$
We have, for every $R,\omega>0$, 
$$\frac{\partial \chi}{\partial R}(\omega, R) = 4R \arcsin{\frac{\omega}{2 R}}\ge 0 \ \ \ \text{and}\ \ \ \ \frac{\partial \chi}{\partial \omega}(\omega, R) = \sqrt{4R^2-\omega^2}\ge 0.$$
Thus, using Theorem \ref{cifrsalth}, we have, for every $t\in [0,r(\Omega))$, 
$$|\Omega_{-t}|\leq \chi(\omega(\Omega_{-t}), R(\Omega_{-t}))\leq \chi(\omega(\Om)-2t ,R(\Om)-t)=:g_4^\Om(t),$$
where we used \eqref{width} and \eqref{circumradius}.
By Lemma \ref{lem:main}, we have 
$$h(\Om)\ge \frac{1}{t_{g_4^\Om}},$$
where $t_{g_4^\Om}$ is the smallest solution to the equation $g_4^\Om(t)=\pi t^2$ on the interval $[0,r(\Om)]$. 

\item Let us now prove the upper bound \eqref{upper_hRw}. We recall that for all $\Om\in\mathcal{K}^2$ such that $\omega(\Om)\le \frac{3}{2}R(\Om)$, we have 
$$\abs{\Om}\ge \abs{T_I}\ \ \text{and}\ \ \ r(\Om)\ge r(T_I),$$
 where $T_I$ is a subequilateral triangle such that $\omega(T_I)=\omega(\Om)$ and $R(T_I)=R(\Om)$, see respectively \eqref{ARw_low} and \eqref{cifre_gomis_1_eq}. 

 To conclude, one just has to combine those estimates with the inequality  
$$  h(\Om)\le \frac{1}{r(\Om)}+ \sqrt{\frac{\pi}{\abs{\Om}}}, $$
see \cite{ftJMAA}, which is an equality for sets that are homothetic to their form bodies, in particular, subequilateral  triangles. 

\end{itemize}
\end{proof}

\begin{figure}[h]
    \centering
    \includegraphics[scale=.32]{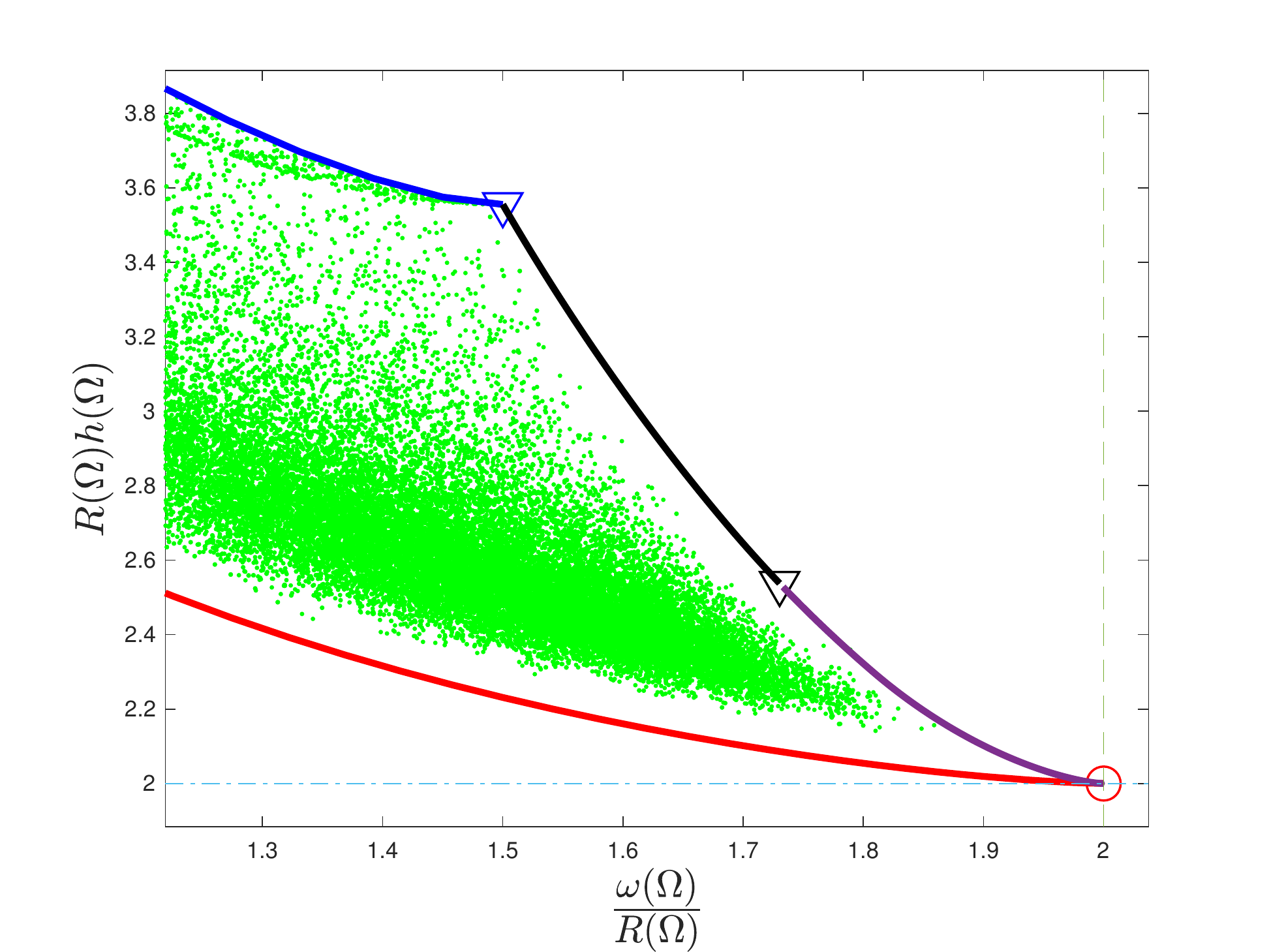}
    \includegraphics[scale=.32]{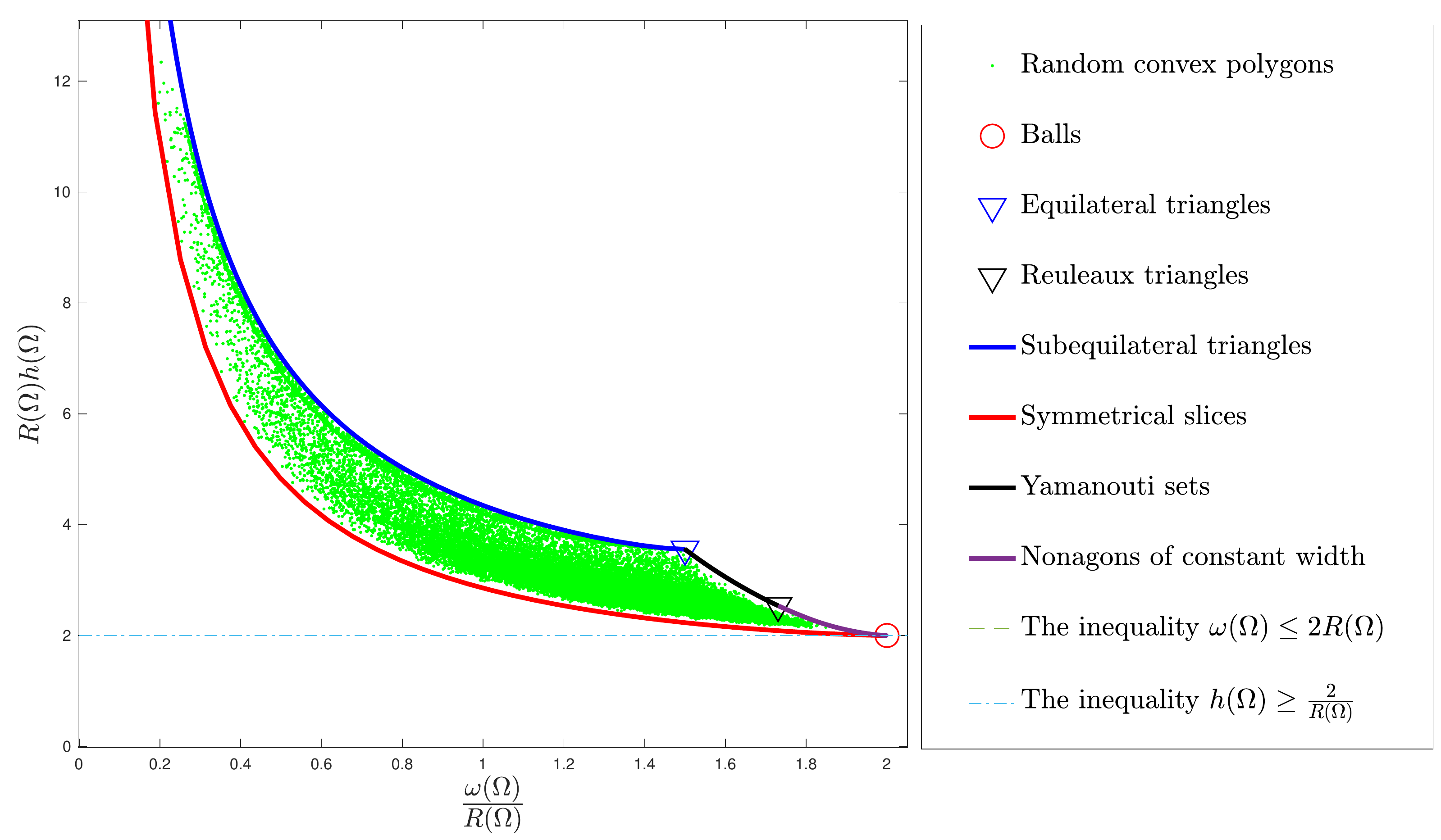}
    \caption{Blaschke--Santal\'o diagram of the triplet $(\omega,h,R)$.}
    \label{fig:hwR}
\end{figure}

\newpage
In the following remark, we give some explicit sharp bounds, that do not correspond to parts of the boundary of the Blaschke--Santal\'o diagram.

\begin{remark}\label{rk:hwR}
We can prove that
\begin{equation}\label{ns}
     h(\Om)\le \frac{3}{\omega(\Om)}+\sqrt{\frac{\pi}{\sqrt{3}R(\Om)\omega(\Om)}}.
\end{equation}
We recall the following inequalities, proved in \cite{inequalities_convex},
\begin{equation}
    \label{Arw}
    \abs{\Om}\ge \sqrt{3}R(\Om)\omega(\Om)\ \ \ \text{and}\ \ \ 
    \omega(\Om)\le 3 r(\Om).
\end{equation}
By combining these estimates and the upper bound in \eqref{eq:hra}, we have
$$h(\Om)\leq \frac{3}{\omega(\Om)}+\sqrt{\frac{\pi}{\sqrt{3}R(\Om)\omega(\Om)}}.$$
Since the equality in \eqref{eq:hra} is achieved by sets that are homothetic to their form bodies, while the equalities in \eqref{Arw} are achieved by equilateral triangles, that are particular sets that are homothetic to their form bodies, we have the equality in \eqref{ns} for equilateral triangles.

Moreover, another sharp lower bound can be obtained by using the strategy from Lemma \ref{lem:main}, starting from
$$\abs{\Om}< 2 R(\Om)\omega(\Om),$$
which is asymptotically achieved by a sequence of rectangles or spherical slices with circumradius that goes to infinity (see \cite{henk}), as shown in Figure 
\ref{fig:remark54}. We get
$\abs{\Om_{-t}}\le 2 R(\Om_{-t})\omega(\Om_{-t})\le 2(R(\Om)-t)(\omega(\Om)-2t)$
and, consequently, 
\begin{equation}\label{lrr}
   h(\Om)\ge\frac{4-\pi}{(2R(\Om)+\omega(\Om))-\sqrt{(2R(\Om)+\omega(\Om))^2-(4-\pi)(2R(\Om)\omega(\Om))}}.  
\end{equation}
\begin{figure}[h]
    \centering
    \includegraphics[scale=.6]{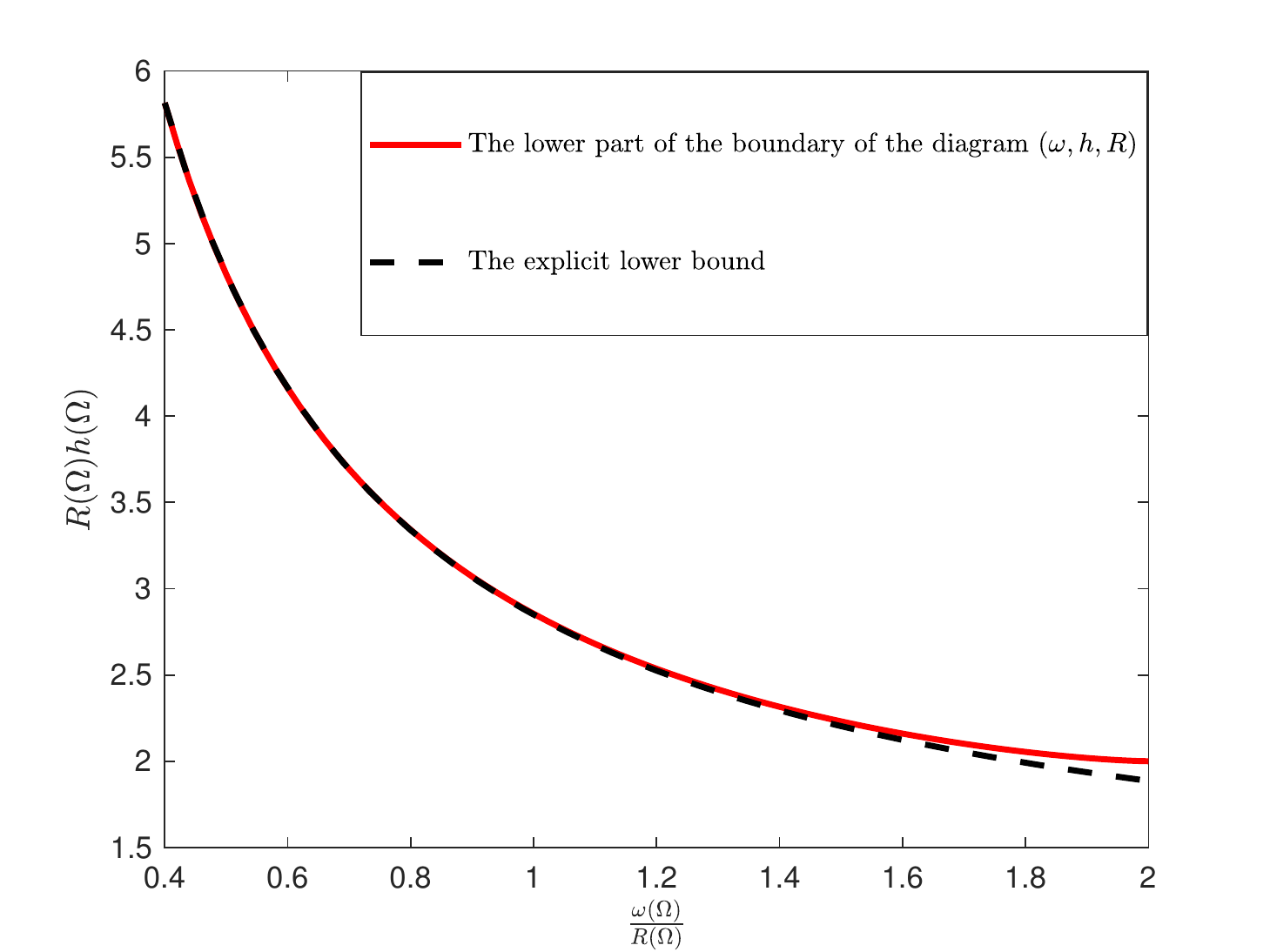}
    \caption{The explicit inequality \eqref{lrr} and the lower part of the boundary of the diagram $(\omega,h,d)$.}
    \label{fig:remark54}
\end{figure}
\end{remark}

\newpage
\subsection{The triplet $(\omega, h, |\cdot|)$}
\begin{proposition}\label{prop_hAw}
Let $\Omega\in\mathcal{K}^2$. We have
\begin{equation}\label{hAw_low}
    h(\Om)\geq \frac{2}{\omega(\Om)}+\frac{\pi \omega(\Om)}{2 \abs{\Om} },
\end{equation}
where the equality is achieved by stadiums. Moreover, we have
\begin{equation}\label{hAw_up}
    h(\Om)\leq h(T_I),
\end{equation}
where $T_I$ is a subequilateral triangle such that $|\Om|=|T_I|$ and $\omega(\Om)=\omega(T_I)$. The equality in \eqref{hAw_up} is achieved by the subequilateral triangle $T_I$.
\end{proposition}

\begin{proof}
\begin{itemize}
    \item Let us prove inequality \eqref{hAw_low}. We recall the lower bound in \eqref{eq:hra}
$$h(\Om)\ge \frac{1}{r(\Om)}+\frac{\pi r(\Om)}{2|\Om|},$$
which is an equality if and only if $\Om$ is a stadium. Inequality \eqref{hAw_low} is then a consequence of the fact that the function $r\longmapsto  \frac{1}{r}+\frac{\pi r}{2|\Om|}$ is strictly decreasing and $r(\Om)\leq \frac{\omega(\Om)}{2}$ (where the equality holds for stadiums).


\item Let us now prove inequality \eqref{hAw_up}. We start by  recalling  inequality \eqref{cifre_salinas_rwA}:
$$    |\Om|\le \sqrt{\frac{ r^4(\Om)\omega^3(\Om)}{(\omega(\Om)-2r(\Om))^2(4r(\Om)-\omega(\Om))}}=:f_{\omega(\Om)}(r(\Om)).$$
By direct computations, we prove that the continuous function 
$$f_{\omega(\Om)}: r\longmapsto \sqrt{\frac{r^4 \omega(\Om)^3}{(\omega(\Om)-2r)^2(4r-\omega(\Om))}}$$
is strictly increasing on $\left[\frac{\omega(\Om)}{3}, \frac{\omega(\Om)}{2}\right)$. Let us denote by $g_{\omega(\Om)}$ the inverse function of $f_{\omega(\Om)}$, which is also continuous and strictly increasing. We have 
$$r(\Om)\ge g_{\omega(\Om)}(|\Om|)=r(T_I),$$
where $T_I$ is any subequilateral triangle such that $\omega(T_I) = \omega(\Om)$ and $|T_I|=|\Om|$. Thus, we have 
$$h(\Om)\leq \frac{1}{r(\Om)}+\sqrt{\frac{\pi}{|\Om|}}\leq\frac{1}{r(T_I)}+\sqrt{\frac{\pi}{|T_I|}}=h(T_I),$$with equality if and only if $\Om=T_I$, where we used the upper bound in \eqref{eq:hra}.
\end{itemize}
 
\end{proof}

\begin{figure}[h]
    \centering
   \includegraphics[scale=.33]{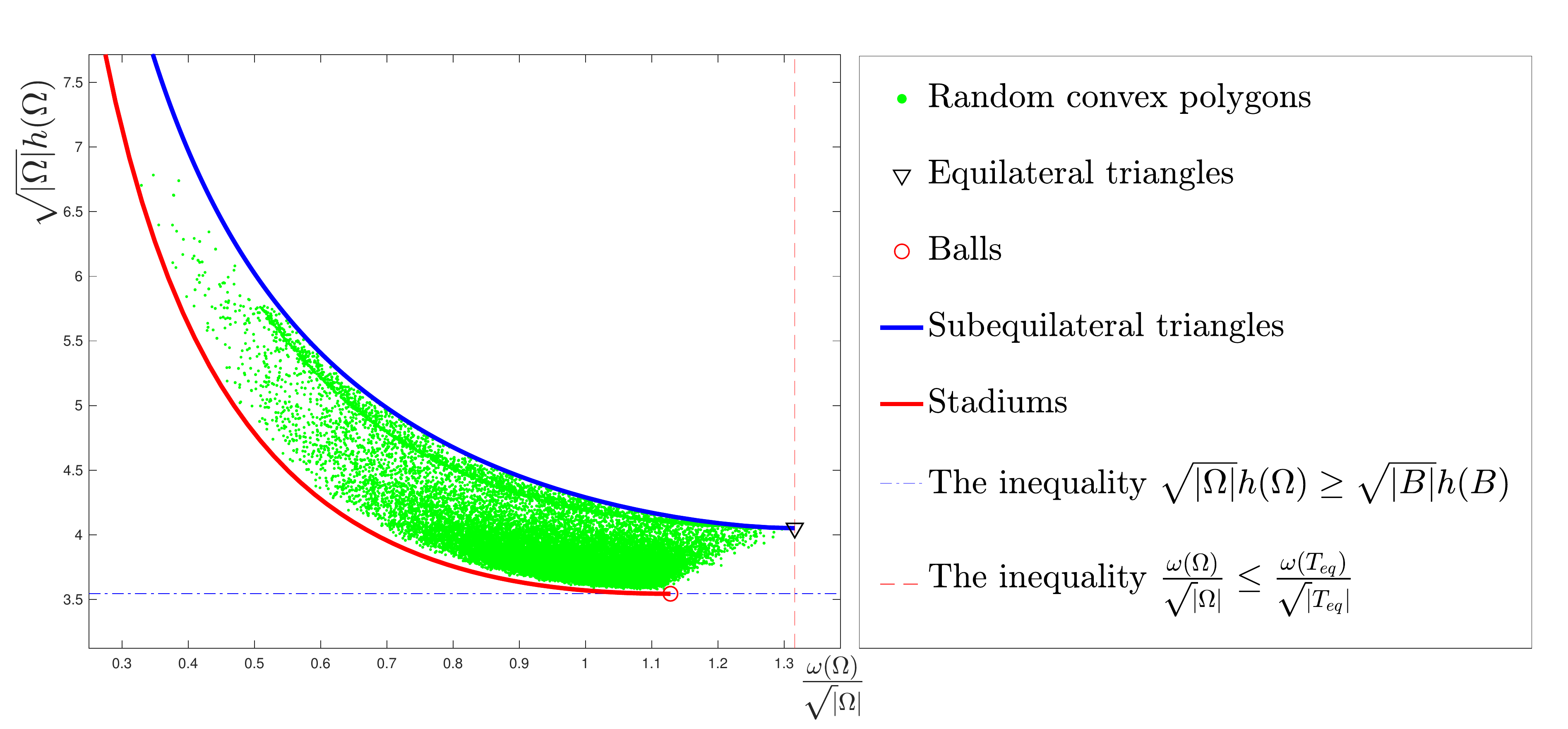}
    \caption{Blaschke--Santal\'o diagram of the triplet $(\omega,h,|\cdot|)$.}
    \label{fig:hwa}
\end{figure}

\begin{remark}\label{rk:hAw}
One may use classical convex geometry inequalities to obtain simpler bounds than the implicit one given in \eqref{hAw_up}. Indeed, if we combine the inequalities in \eqref{eq:hra} and the following classical ones 
\begin{equation*}\label{eq:rw}
    \frac{2}{\omega(\Om)}\leq \frac{1}{r(\Om)}\leq \frac{2}{\omega(\Om)}+ \frac{\omega(\Om)}{\sqrt{3}|\Om|},
\end{equation*}
where the lower bound is realized in particular by stadiums, and the upper one by equilateral triangles (see for example \cite{inequalities_convex} and the references therein), we obtain the following inequalities
\begin{equation}\label{hAw_up1}
h(\Om)\leq \frac{2}{\omega(\Om)}+ \frac{\omega(\Om)}{\sqrt{3}|\Om|}+\sqrt{\frac{\pi}{|\Om|}}
\end{equation}
and 
\begin{equation}\label{hAw_up2}
    h(\Om)\leq \frac{2}{\omega(\Om)-\frac{\omega(\Om)^3}{4|\Om|}}+\sqrt{\frac{\pi}{|\Om|}}.
\end{equation}
The bound \eqref{hAw_up1} is attained for equilateral triangles and \eqref{hAw_up2} is asymptotically attained for a sequence of thin subequilateral triangles. 
\end{remark}

\subsection{The triplet $(\omega, h, P)$}
\begin{proposition}\label{prop_hwP}
Let $\Omega\in\mathcal{K}^2$. We have 
\begin{equation}
    \label{hwp_low}
    h(\Om)\ge \frac{2}{\omega(\Omega)}+\frac{2\pi}{2P(\Om)-\pi \omega(\Om)},
\end{equation}
where the equality is achieved by stadiums. 

Moreover, if $P(\Om)\ge 2\sqrt{3}\omega(\Om)$, then,
\begin{equation}
    \label{hwp_up}
    h(\Om)\le h(T_I),
\end{equation}
where $T_I$ is a subequilateral triangle such that $P(T_I)=P(\Om)$ and $\omega(T_I)=\omega(\Om)$. The equality in \eqref{hwp_up} is achieved by the subequilateral triangle $T_I$.
\end{proposition}

\begin{proof}
\begin{itemize}
    \item Inequality \eqref{hwp_low} is a consequence of \eqref{hAw_low} and the inequality
$$|\Om|\leq \frac{\omega(\Om)}{2}\left(P(\Om)-\frac{\pi\omega(\Om)}{2}\right),$$
which is an equality for stadiums, see for example \cite{inequalities_convex}. 

\item Let us now assume that $P(\Om)\ge 2\sqrt{3}\omega(\Om)$. In order to prove \eqref{hwp_up}, we recall inequality \eqref{cifre_salinas_rwP}:

$$    P(\Om)\le \sqrt{\frac{4 r^2(\Om)\omega^3(\Om)}{(\omega(\Om)-2r(\Om))^2(4r(\Om)-\omega(\Om))}}=:f_{\omega(\Om)}(r(\Om)).$$
By direct computations, we prove that the continuous function 
$$f_{\omega(\Om)}: r\longmapsto \sqrt{\frac{4 r^2 \omega(\Om)^3}{(\omega(\Om)-2r)^2(4r-\omega(\Om))}}$$
is strictly increasing on $\left[\frac{\omega(\Om)}{3}, \frac{\omega(\Om)}{2}\right)$. Let us denote by $g_{\omega(\Om)}$ the inverse function of $f_{\omega(\Om)}$, which is also continuous and strictly increasing. We have 
$$r(\Om)\ge g_{\omega(\Om)}(P(\Om))=r(T_I),$$
where $T_I$ is any subequilateral triangle such that $\omega(T_I) = \omega(\Om)$ and $P(T_I)=P(\Om)$. Moreover, since $P(\Om)\ge 2\sqrt{3}\omega(\Om)$, we have by the results contained in  \cite{yamanouti},
$$|\Om|\ge |T_I|,$$
(see also \cite{inequalities_convex} as a reference).
Finally, we obtain $$h(\Om)\leq \frac{1}{r(\Om)}+\sqrt{\frac{\pi}{|\Om|}}\leq \frac{1}{r(T_I)}+\sqrt{\frac{\pi}{|T_I|}}=h(T_I).$$

\end{itemize}
 
\end{proof}

\begin{figure}[h]
    \centering
    \includegraphics[scale=.31]{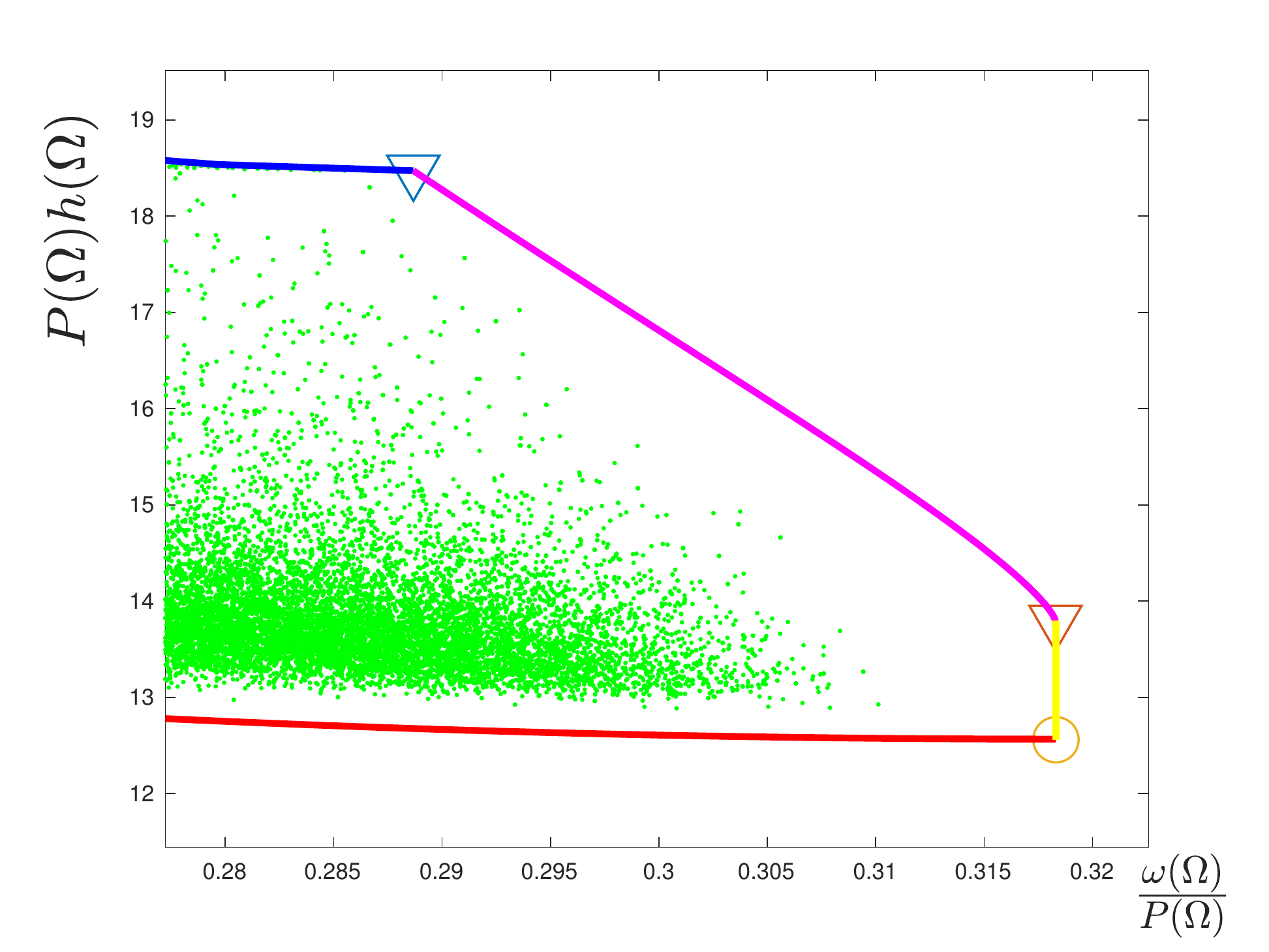}
    \includegraphics[scale=.3]{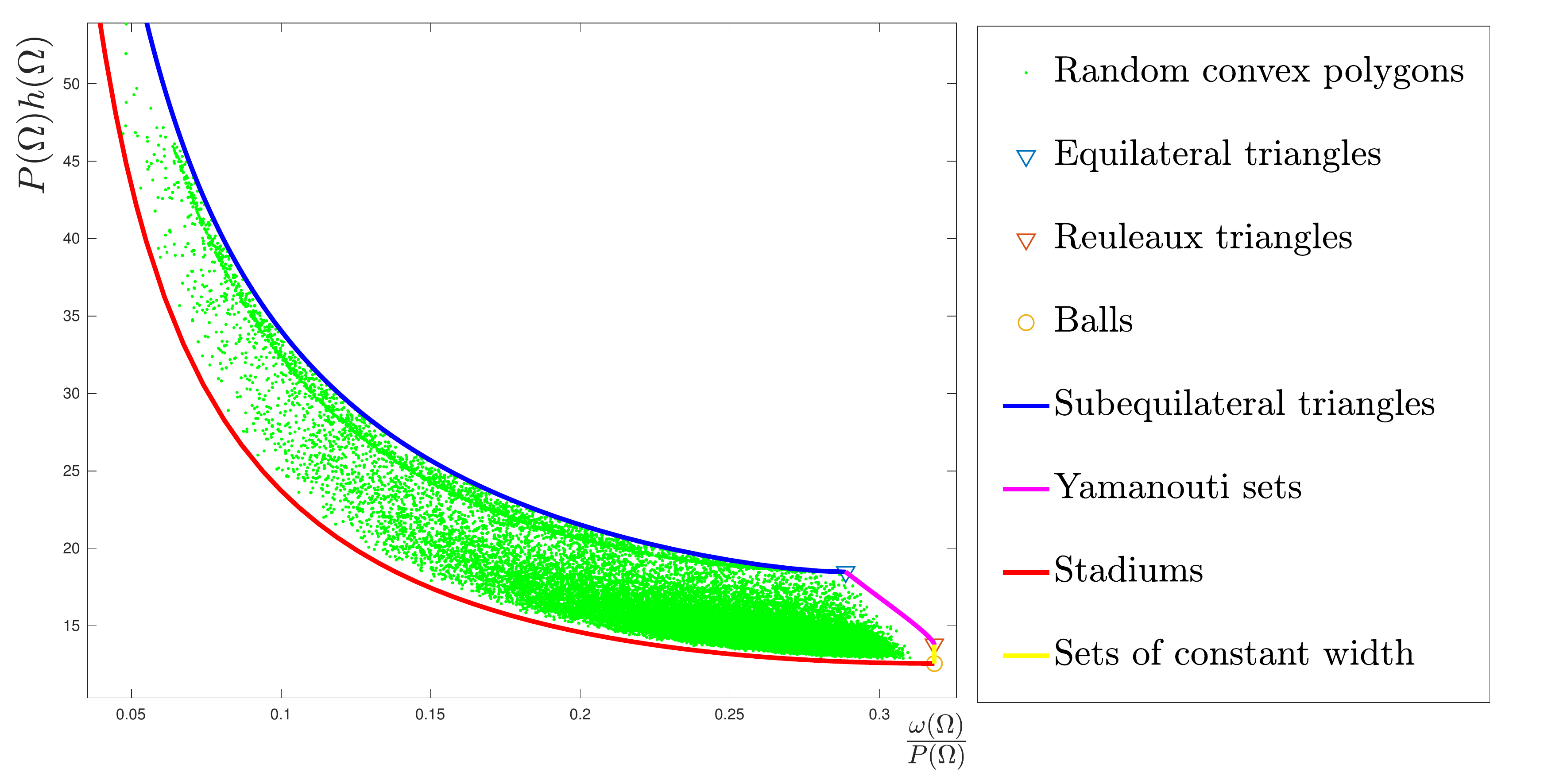}
    \caption{Blaschke--Santal\'o diagram of the triplet $(\omega,h,P)$.}
    \label{fig:hpw}
\end{figure}

\subsection{The triplet $(R, h, d)$}
\begin{proposition}\label{prop_hRd}
Let $\Omega\in\mathcal{K}^2$. We have 
\begin{equation}\label{hRdup}
    h(\Om)\le \frac{2R(\Om) \left(2R(\Om)+ \sqrt{4R(\Om)^2-d(\Om)^2}\right)}{d(\Om)^2\sqrt{4R(\Om)^2-d(\Om)^2}}+ \sqrt{\frac{4\pi R(\Om)^2}{d(\Om)^3 \sqrt{4R(\Om)^2-d(\Om)^2}}},
\end{equation}
where the equality is achieved by subequilateral triangles. 
\end{proposition}
\begin{proof}
Inequality \eqref{hRdup} is obtained by combining 
$$h(\Om)\leq \frac{1}{r(\Om)}+\sqrt{\frac{\pi}{|\Om|}},$$
see \cite{ftJMAA}, which is an equality for sets that are homothetic to their form bodies (in particular subequilateral triangles), and 
$$    |\Om|\ge \frac{d(\Om)^3\sqrt{4R(\Om)^2-d(\Om)^2}}{4R(\Om)^2}
    \ \ \text{and}\ \ \  r(\Om)\ge \frac{d(\Om)^2\sqrt{4R(\Om)^2-d(\Om)^2}}{2R(\Om)\left(2R(\Om)+\sqrt{4R(\Om)^2-d(\Om)^2}\right)},
$$
respectively proved in \cite{cifre2} and \cite{santalo}, where the equalities hold only for subequilateral triangles.
\end{proof}

\begin{figure}[h]
    \centering
    \includegraphics[scale=.4]{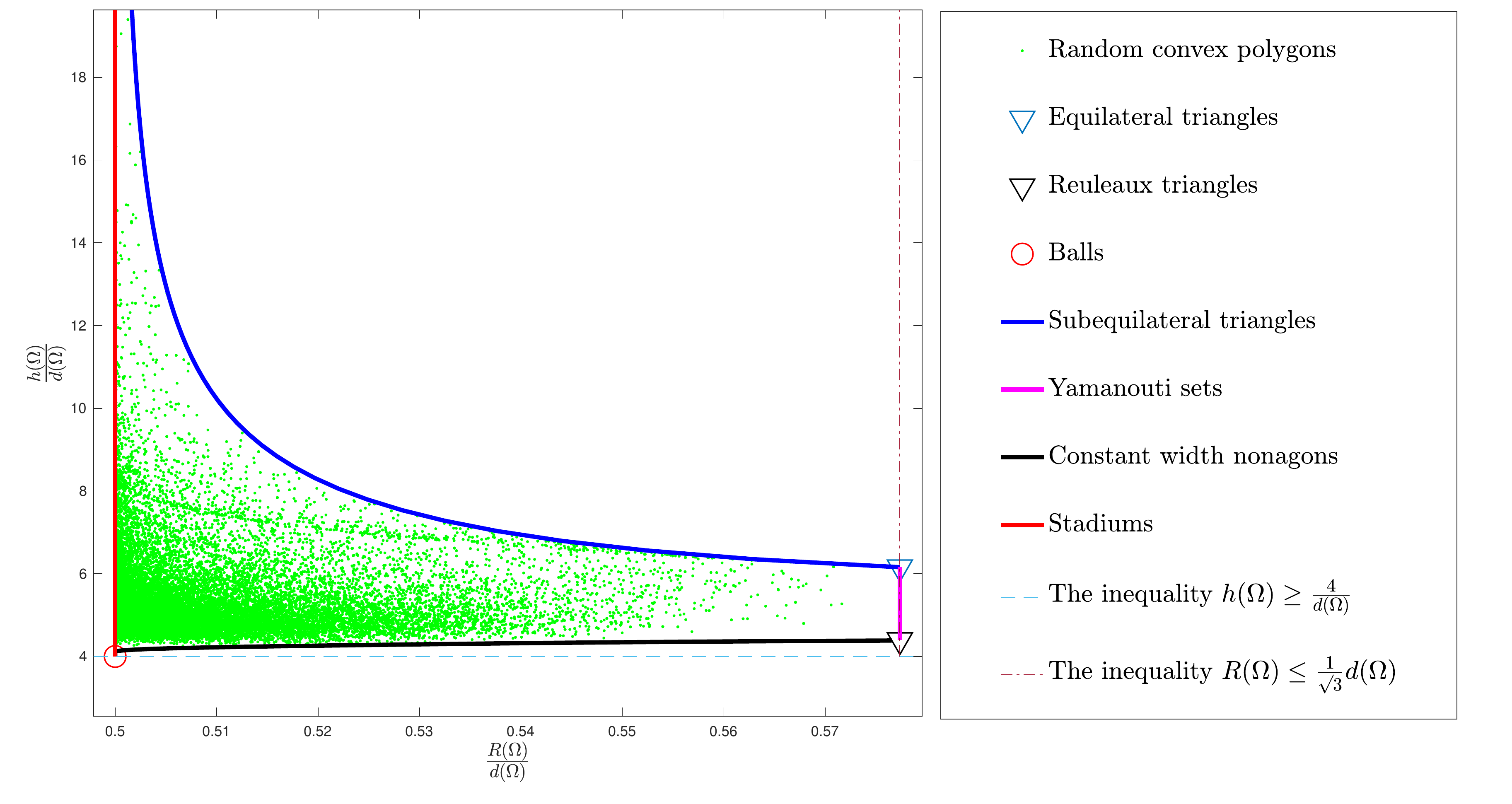}
    \caption{Blaschke--Santal\'o diagram of the triplet $(R, h,d)$, See Conjecture \ref{conj:hdR} and Figure \ref{cifrefig} for the definition of constant width nonagons.}
    \label{fig:hdR}
\end{figure}

\subsection{The triplet $(\omega, h, r)$}
\begin{proposition}\label{prop_hwr}
Let $\Omega\in\mathcal{K}^2$. We have 
\begin{equation}\label{ineq:hwr_lower}
h(\Omega)\ge \frac{1}{r(\Omega)}+\frac{1}{r(\Omega)}\sqrt{\pi\left(1-\frac{2r(\Omega)}{\omega(\Omega)}\right)\sqrt{\frac{4r(\Omega)}{\omega(\Omega)}-1}},    \end{equation}
where the equality is achieved by subequilateral triangles.
\end{proposition}

\begin{proof}
The present demonstration is inspired by the proof of \cite[Theorem 5]{cifre_salinas}. 

It is known that the incircle of a set $\Om$ meets the boundary of $\Om$ either in two points contained in two parallel lines, or in (at least) three points that form the vertices of a triangle, see \cite{bonnesen}. In the first case, we have $\omega(\Om)=2r(\Om)$, thus inequality \eqref{ineq:hwr_lower} is equivalent to $h(\Om)\ge \frac{1}{r(\Om)}$, which is proved in Proposition \ref{two}. In the second case, let us denote by $T$ a triangle formed by three lines of support common to $\Om$ and the incircle. 
We have $r(\Om)=r(T)$ and, by $\Omega \subset T$ and the monotonicity with respect to the inclusion, we get
\begin{equation}\label{h}
   h(\Om)\ge h(T)
\end{equation}
and 
\begin{equation}\label{omega}
\omega(\Om)\leq \omega(T).
\end{equation}
So, inequality \eqref{ineq:hwr_lower} is equivalent to the following one
$$
\frac{1}{r(\Om)h(\Om)}f\left(\frac{r(\Om)}{\omega(\Om)}\right)\leq 1,    
$$
where $f:x\in\left[\frac{1}{3},\frac{1}{2}\right]\longmapsto 1+\sqrt{\pi\left(1-2x\right)\sqrt{4x-1}}$. We observe that the function $g:x\in\left[\frac{1}{3},\frac{1}{2}\right]\longmapsto \left(1-2x\right)\sqrt{4x-1}$ is decreasing. Indeed,
$$g'(x) = \frac{4(1-3x)}{\sqrt{4x-1}}\leq 0.$$
Thus, $f$ is also decreasing on $\left[\frac{1}{3},\frac{1}{2}\right]$. Then, since $\frac{r(\Om)}{\omega(\Om)}\ge \frac{r(T)}{\omega(T)}$ by \eqref{omega}, we have 
$$f\left(\frac{r(\Om)}{\omega(\Om)}\right) \leq  f\left(\frac{r(T)}{\omega(T)}\right).$$
Moreover, we get by \eqref{h}, $$\frac{1}{r(\Om)h(\Om)}\leq \frac{1}{r(T)h(T)}.$$
Thus, we obtain
$$\frac{1}{r(\Om)h(\Om)}f\left(\frac{r(\Om)}{\omega(\Om)}\right) \leq \frac{1}{r(T)h(T)}f\left(\frac{r(T)}{\omega(T)}\right)=\frac{1}{1+r(T)\sqrt{\frac{\pi}{|T|}}}f\left(\frac{r(T)}{\omega(T)}\right),$$
where we used the equality $h(T) = \frac{1}{r(T)}+\sqrt{\frac{\pi}{|T|}}$, which holds because $T$ is a triangle, see \cite[Theorem 1.3]{ftJMAA}. 
Now, we use the inequality $$|T|\leq \frac{r(T)^2}{\left(1-\frac{2r(T)}{\omega(T)}\right)\sqrt{\frac{4r(T)}{\omega(T)}-1}},$$
which is an equality if and only if $T$ is a subequilateral triangle, see \cite[Theorem 5]{cifre_salinas}. We then have
$$\frac{1}{r(\Om)h(\Om)}f\left(\frac{r(\Om)}{\omega(\Om)}\right) \leq\frac{1}{1+r(T)\sqrt{\frac{\pi}{|T|}}}f\left(\frac{r(T)}{\omega(T)}\right)\leq 1,$$
which ends the proof. 
\end{proof}

\begin{figure}[h]
    \centering
    \includegraphics[scale=.4]{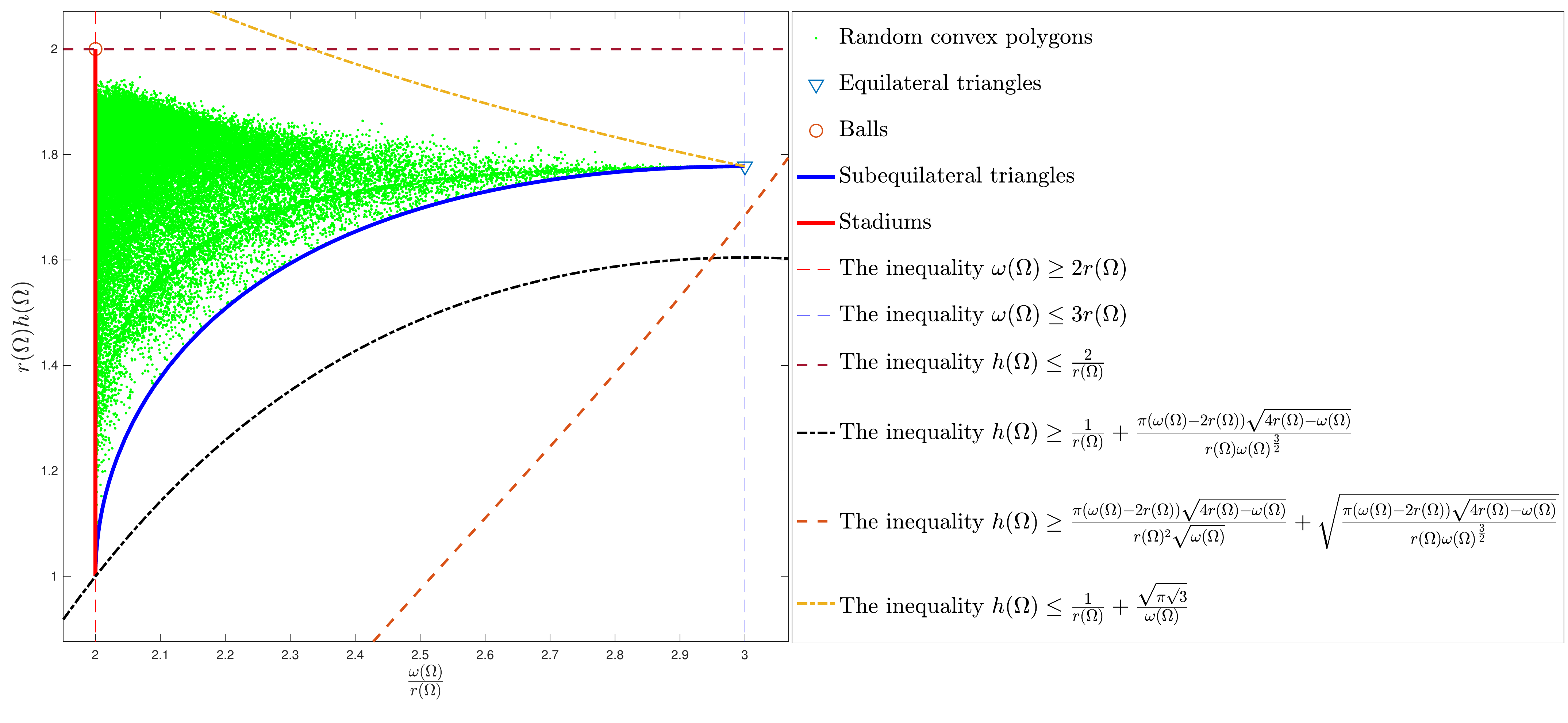}
    \caption{Blaschke--Santal\'o diagram of the triplet $(\omega,h, r)$.}
    \label{fig:hwr}
\end{figure}

\begin{remark}It is possible to prove the following upper bound 
\begin{equation*}\label{ineq:hwr_upper}
h(\Omega) \leq \frac{1}{r(\Omega)}+ \frac{\sqrt{\pi\sqrt{3}}}{\omega(\Omega)},
\end{equation*}
where the equality holds for equilateral triangles, see Figure \ref{fig:hwr}. 

 In order to prove it, one may combine the upper bound in \eqref{eq:hra} and the inequality $\omega(\Om)^2\leq \sqrt{3}|\Om|$, see \cite{inequalities_convex}. 
\end{remark}

\section{Conclusions and conjectures}\label{secult}

In this Section, we collect the conjectures that we deduced from the numerical approximation of the Blaschke--Santaló diagrams that we have plotted in the previous sections.

\begin{conjecture}
We consider the diagram $(\omega, h,  d)$ plotted in Figure \ref{fig:hwd}. Let $\Om\in \mathcal{K}^2$, we conjecture that if $\frac{\sqrt{3}}{2} d(\Omega)\le\omega(\Omega)\le d(\Omega)$, then 
$$h(\Omega)\le h(Y)$$
where $Y$ is a Yamanouti set (see  Definition \ref{yamanouti}) such that $\omega(Y)=\omega(\Om) $ and $d(Y)=d(\Om)$.
\end{conjecture}

\begin{conjecture} We consider the diagram $(\omega,h, P)$ plotted in Figure \ref{fig:hpw}. Let $\Om\in \mathcal{K}^2$, we conjecture that, if $\pi\omega(\Om)\le P(\Om)\le 2\sqrt{3}\omega(\Om)$, then
\begin{equation*}
    h(\Omega)\leq h(Y),  
\end{equation*}
where $Y$ is a Yamanouti set (see  Definition \ref{yamanouti}) such that $P(Y)=P(\Om)$ and $\omega(Y)=\omega(\Om)$.
\end{conjecture}

\begin{conjecture}\label{conj:hdR} We consider the diagram $(R, h, d)$ plotted in Figure \ref{fig:hdR}. Let $\Om\in\K^2$, we conjecture that
\begin{equation*}
    h(\Omega)\geq h(N),  
\end{equation*}
where $N$ refers to a nonagon of constant width such that $d(N) = d(\Om)$ and $R(N)=R(\Om)$ described in \cite{cifre2} as follows:
let $\Gamma$ and $\gamma$ be the circumcircle and the incircle of a constant width set $K$ that are known to be concentric and such that $d(K)=\omega(K)= R(K) + r(K)$. The extremal set can be constructed in the following way: an equilateral triangle $ABC$ is inscribed in the circle $\Gamma$, and now we take the circular arcs of radius $R(K)+r(K)$ drawn about the three vertex points. These arcs touch $\gamma$ at the opposite points $\Bar{A}, \Bar{B}, \Bar{C}$ of $A, B, C$, respectively. Furthermore, we construct three circles of radius $(R(K) + r(K))/2$ that have the sides of the triangle as chords and whose centers lie inside the triangle. The required constant width set has 3-fold symmetry and is formed by nine arcs
of the six constructed circles, see Figure \ref{cifrefig}.
\end{conjecture}

\begin{figure}[h]
    \centering
    \includegraphics[scale=.35]{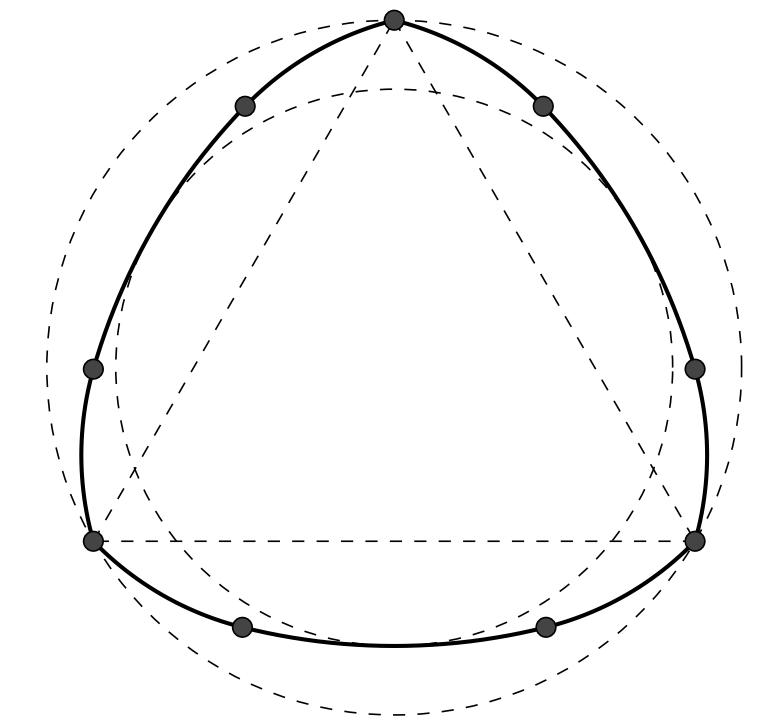}
    \caption{A nonagon of constant width.}
    \label{cifrefig}
\end{figure}

\begin{conjecture} We consider the diagram $(\omega, h,  R)$ plotted in Figure \ref{fig:hwR}. Let $\Om\in \K^2$, we conjecture that:
\begin{itemize}
    \item If $\omega(\Om)\in [\frac{3}{2}R(\Om),\sqrt{3}R(\Om)]$, then 
\begin{equation*}
    h(\Omega)\leq h(Y),  
\end{equation*}
where $Y$ is a Yamanouti set (see  Definition \ref{yamanouti}) such that $\omega(Y) = \omega(\Om)$ and $R(Y)=R(\Om)$. 
    \item If $\omega(\Om)\ge \sqrt{3}R(\Om)$, then
    $$h(\Om)\leq h(N),$$
    where $N$ refers to a nonagon of constant width such that $\omega(N) = \omega(\Om)$ and $R(N)=R(\Om)$.
\end{itemize}
\end{conjecture}

\newpage

\section{Appendix: Summary tables with the results}
In this first table, we summarize the results relative to the diagrams that are completely solved.
\smaller 
\begin{figure}[!h] \label{table1}
    \centering
\begin{tabular}{|c|c|c|c|c|c|}
\hline
Param. & Condition & Inequality  & Extremal sets & Ref. \tabularnewline
\hline 
$P,h, A$ &  & $\displaystyle h\leq \frac{P}{A}$ & Cheeger of itself & \cite{ftouhi_cheeger}  \tabularnewline
 &  & $\displaystyle h\ge \frac{P+\sqrt{4\pi A}}{2A}$ &  sets that are homothetic to their form bodies &   \cite{ftouhi_cheeger}  \tabularnewline[1ex] \hline
$r, h,A$ &  & $\displaystyle h\leq \frac{1}{r}+\sqrt{\frac{\pi}{A}}$ &  sets that are homothetic to their form bodies & \cite{ftJMAA}  \tabularnewline
&  & $\displaystyle h\ge \frac{1}{r}+\frac{\pi r}{A}$ &  stadiums &  \cite{ftJMAA} \tabularnewline[1ex] \hline
$P,h, r$ &  & $\displaystyle {h\leq \frac{1}{r}+\sqrt{\frac{2\pi }{P r}}}$  & sets that are homothetic to their form bodies & Prop. \ref{prop_hrP}  \tabularnewline
 &  & $\displaystyle {h\ge \frac{1}{r}+\frac{\pi }{P -\pi r}}$ &  stadiums &  Prop. \ref{prop_hrP} \tabularnewline[1ex]
\hline
$d,h,r$  &  & $\displaystyle {h\leq \frac{1}{r}+ \sqrt{\frac{\pi}{r\sqrt{d^2-4r^2}+r^2(\pi-2\arccos{\left(\frac{2r}{d}\right)})}}}$ & two-cup bodies & Prop. \ref{prop_hdr} \tabularnewline[1ex]
 &  & ${h\ge \frac{1}{t_{g_1^\Om}} }$    \textcolor{blue}{$(i)$} & spherical slices/smoothed nonagons & Prop. \ref{prop_hdr} 
 \tabularnewline[1ex]
  &  & $ h > \dfrac{4-\pi}{d+2 r-\sqrt{(d+2r)^2-2(4-\pi)dr}} $ & thinning rectangles & Remark \ref{rk:hdr} \tabularnewline[2ex]
\hline
$ R,h,r$ &  & $\displaystyle h\le \frac{1}{r}+ \sqrt{\frac{\pi }{2r\left(\sqrt{R^2-r^2}+r\arcsin\left(\frac{r}{R}\right)\right)}}$ &  two-cup bodies & Prop. \ref{prop_hRr}  \tabularnewline
 &  & $\displaystyle h\ge\frac{1}{t_{g_2^\Om}}$ \textcolor{blue}{$(ii)$} &  spherical slices &  Prop. \ref{prop_hRr} \tabularnewline[1ex]
  &  & $h > \frac{4-\pi}{2(R+r)-\sqrt{4(R+r)^2-4(4-\pi)Rr}}$  &  thinning rectangles &  Remark \ref{rk:hrR} \tabularnewline[2ex]
\hline
\end{tabular}
    \label{fig:polygons}
\end{figure}
\normalsize
\begin{footnotesize}
\newline
\textcolor{blue}{$(i)$} $t_{g_1^\Om}$ is the smallest solution on $[0, r(\Om)]$ to $$g_1^\Om(t):=\psi(d(\Om)-2t,r(\Om)-t)=\pi t^2, $$ where
\begin{equation*} 
   \psi(d,r):=\begin{cases}
    \displaystyle{\frac{3\sqrt{3}r}{2}(\sqrt{d^2-3r^2}-r)+\frac{3d^2}{2}\left(\frac{\pi}{3}-\arccos{\left(\frac{\sqrt{3}r}{d}\right)}\right)}, & \text{if} \, \, \,  d\le r D^*\vspace{1mm} \\
    \displaystyle{r\sqrt{d^2-4r^2}+\frac{d^2}{2}\arcsin{\left(\frac{2r}{d}\right)}}, & \text{if} \, \, \, d\ge r D^*. 
\end{cases} 
\end{equation*}
and $D^*$ is the unique number in $[2,2\sqrt{3}]$ for which the two expression of the function $\psi(d,r)$ are equal.
\newline
 \textcolor{blue}{$(ii)$}  $t_{g_2^\Om}$ is the smallest solution on $[0, r(\Om)]$ to  $${g_2^\Om}(t):=2\left((r-t)\sqrt{(R(\Omega)-t)^2-(r(\Om)-t)^2}+(R(\Om)-t)^2\arcsin{\left(\frac{r(\Om)-t}{R(\Om)-t}\right)}\right)=\pi t^2.$$
\end{footnotesize}

In this second table, we summarize the results of the partially solved Blaschke--Santal\'o diagrams.

\begin{figure}[!h] \label{table3}
    \centering
\begin{tabular}{|c|c|c|c|c|c|}
\hline
Param. & Condition & Inequality  & Extremal sets & Ref. \tabularnewline
\hline
$ \omega, h, d$ & $\omega\le \sqrt{3}/2d$ & $h\le \frac{1}{r(\omega, d)}+\sqrt{\frac{2\pi}{\omega d}} $ \textcolor{blue}{$(iii)$} &  subequilateral triangles & Prop. \ref{prop_hdw}     \tabularnewline
& $\sqrt{3}/2d\le \omega\le d$ & $  h\le \frac{\sqrt{3}}{\sqrt{3}\omega-d}+\sqrt{\frac{2\pi}{\pi \omega^2-\sqrt{3}d^2+6\omega(\tan\left(\arccos(\frac{\omega}{d})\right)-\arccos(\frac{\omega}{d}))}}$  & equilateral triangles  &  \tabularnewline
 &  & $h \ge \frac{1}{t_{g_3^\Om}}$ \textcolor{blue}{$(iv)$}& spherical slices & \tabularnewline
  &  & $    h>\frac{1}{\omega}+\frac{1}{d}+\sqrt{\left(\frac{1}{\omega}+\frac{1}{d}\right)^2-\frac{4-\pi}{\omega d}} $ & thinning rectangles & Remark \ref{rk:hwd} \tabularnewline
\hline
$ \omega,h, R$ & $\omega\le 3/2 R$ & $h\leq \frac{1}{r(\omega, R)}+\sqrt{\frac{\pi}{A(\omega, R)}}$ \textcolor{blue}{$(v)$} & subequilateral triangles & Prop. \ref{prop_hRw}  \tabularnewline
 &  & $h\ge\frac{1}{t_{g_4^\Om}}$ \textcolor{blue}{$(vi)$} & spherical slices &  \tabularnewline
  &  & $   h\ge\frac{4-\pi}{(2R+\omega)-\sqrt{(2R+\omega)^2-2(4-\pi)R\omega}} $  & thinning rectangles & Remark \ref{rk:hwR}  \tabularnewline
\hline
$ \omega, h, P$ &  $P\geq 2\sqrt{3}\omega$ & $h\le \frac{1}{r(\omega, P)}+\sqrt{\frac{\pi}{A(\omega, P)}} $ \textcolor{blue}{$(vii)$} &  subequilateral triangles & Prop. \ref{prop_hwP}  \tabularnewline
 &  & $  h\ge \frac{2}{\omega}+\frac{2\pi}{2P-\pi \omega}$ &  stadiums &  \tabularnewline
\hline
$ \omega,h, A$ &  & $ h\leq \frac{1}{r(\omega, A)}+ \sqrt{\frac{\pi}{A}}$ \textcolor{blue}{$(viii)$} &  subequilateral triangles & Prop. \ref{prop_hAw}  \tabularnewline
 &  & $ h\geq \frac{2}{\omega}+\frac{\pi \omega}{2 A }$ &  stadiums &  \tabularnewline
\hline
$ R,h,d$ &  & $h\leq \frac{2R(2R+\sqrt{4R^2-d^2})}{d^2\sqrt{4R^2-d^2}}+\sqrt{\frac{4\pi R^2}{d^3\sqrt{4R^2-d^2}}}$ &  subequilateral triangles & Prop. \ref{prop_hRd}  \tabularnewline
\hline
$\omega,h, r$ &  & $h\geq \frac{1}{r}+\frac{1}{r}\sqrt{\pi\left(1-\frac{2r}{\omega}\right)\sqrt{\frac{4r}{\omega}-1}}$ &  subequilateral triangles & Prop. \ref{prop_hwr} 
\tabularnewline
\hline

\end{tabular}
    \label{fig:polygons1}
\end{figure}
\begin{footnotesize}
\noindent \textcolor{blue}{$(iii)$} $r(\omega,d)$ is given by 
   \begin{equation*}
        d^2\left(\omega-2r(\omega,d)\right)^2(4r(\omega,d)-\omega)= 4r^4(\omega,d)\;\omega.
    \end{equation*}
 \newline
\textcolor{blue}{$(iv)$} $t_{g_3^\Om}$ is the smallest solution to
$${g_3^\Om}(t):=f(d(\Om)-2t, \omega(\Om)-2t)=\pi t^2, $$
where
$$f(d,w)=\frac{w}{2}\sqrt{d^2-w^2}+\frac{d^2}{2}\arcsin\left(\frac{w}{d}\right)$$
 \newline
\textcolor{blue}{$(v)$} $r(\omega,R)$ is  given by
\begin{equation*}
    \left(4 r(\omega,R)-\omega   \right)\left( \omega-2 r(\omega,R)\right)= \frac{2 r^3(\omega,R)}{R}
\end{equation*}
and $A(\omega,R)$ is given by 
\begin{equation*}
   16A(\omega,R)^6= R^2 \omega^2\left(16 A(\omega,R)^4-R^2\omega^6\right).
\end{equation*}
\newline
\textcolor{blue}{$(vi)$} $t_{g_4^\Om}$ is the smallest solution on $[0, r(\Om)]$ to
$$g_4^\Om(t):= \chi(\omega(\Om)-2t ,R(\Om)-t)=\pi t^2,$$
where $$\chi(\omega, R):=\frac{\omega}{2} \sqrt{4 R^2-\omega^2} + 2R^2 \arcsin{\frac{\omega}{2R}}.$$
\newline
\textcolor{blue}{$(vii)$} $r(\omega, P)$ is given by
\begin{equation*} 
    (\omega-2r(\omega, P))^2(4r(\omega, P)-\omega)P^2= 4 r(\omega, P)^2\omega^3
\end{equation*}
and $A(\omega, P)$ is the middle root of the equation
 $$128 P A(\omega, P)^3 - 16\omega(5P^2+\omega^2)A(\omega, P)^2 + 16\omega^2 P^3 A(\omega, P)- \omega^3P^4=0$$
\newline 
\textcolor{blue}{$(viii)$} $r(\omega, A)$ is given by 
\begin{equation*} 
    (\omega-2r(\omega, A))^2(4r(\omega, A)-\omega)A^2= r^4(\omega,A) \;\omega^3.
\end{equation*}
\end{footnotesize}

\newpage
Finally, in this last table, we have  summarized the inequalities that we have found and that do not correspond to parts of the boundaries of the corresponding Blaschke--Santal\'o diagrams.

\begin{figure}[!h] \label{table5}
    \centering
\begin{tabular}{|c|c|c|c|c|c|}
\hline
Param. & Condition & Inequality  & Extremal sets & Ref. \tabularnewline
\hline 
$R,h,A$ &  & $   h<\frac{1}{R}+\frac{4R}{A}$ &  thinning rectangles & \cite{ftouhi_cheeger,inequalities_convex}      \tabularnewline
 &  &   $ h\ge \frac{1}{2R}+\frac{\pi R}{2A}+\sqrt{\frac{\pi}{A}}$ & balls & \tabularnewline
\hline
$P,h, R$ & & $ h<\frac{P}{R(P-4R)}$ & thinning rectangles  & \cite{ftouhi_cheeger, santalo}      \tabularnewline
 & &  $  h \ge \frac{4\arcsinc{\left(\frac{4R}{P}\right)}}{P-4R\cos{\left(\arcsinc{\left(\frac{4R}{P}\right)}\right)}} + \sqrt{\frac{8\pi \arcsinc{\left(\frac{4R}{P}\right)}}{P\left(P-4R\cos{\left(\arcsinc{\left(\frac{4R}{P}\right)}\right)}\right)}}$ \textcolor{blue}{$(ix)$}& balls&  \tabularnewline
\hline

$ P,h, d$ & $2d<P<3d$ & $  h < \frac{4}{P-2d} + \sqrt{\frac{4\pi}{(P-2d)\sqrt{P(4d-P)}}}$ &   & \cite{ftJMAA, ftouhi_cheeger, inequalities_convex}  
\tabularnewline
& $3 d\leq P\leq \pi d$ & $   h < \frac{4}{P-2d} + \sqrt{\frac{4\pi}{\sqrt{3} d(P-2d)}}$  &   &  \tabularnewline
 &  &  $ h \ge \frac{4 \arcsinc{\left(\frac{2d}{P}\right)}}{P-2d\cos{\left(\arcsinc{\left(\frac{2d}{P}\right)}\right)}} + \sqrt{\frac{8\pi\arcsinc{\left(\frac{2d}{P}\right)}}{(P-2d\cos{\left(\arcsinc{\left(\frac{2d}{P}\right)}\right)}}}$ \textcolor{blue}{$(ix)$}& balls & \tabularnewline
\hline
$ d,h,A$ && $ h\leq  \frac{4}{d}+\frac{2d}{A}$ &   & \cite{ftJMAA, ftouhi_cheeger, inequalities_convex}       \tabularnewline
&  & $  h< \frac{2 d}{A}+\sqrt{\frac{\pi}{A}}$  &   &  \tabularnewline
 &  &  $h> \frac{d}{A}+\sqrt{\frac{\pi}{A}}$& thinning two-cup & \tabularnewline
\hline

\end{tabular}
\end{figure}

\begin{footnotesize}
\noindent\textcolor{blue}{$(ix)$} where the cardinal-arcsine  is defined as
\begin{equation*}\label{sinc}
    {\rm sinc}:x\in\mathbb{R}\longmapsto {\rm sinc}(x)=\frac{\sin(x)}{x}, \quad{\rm arcsinc}(x)={\rm sinc}^{-1}(x)
\end{equation*}
\end{footnotesize}

\medskip\noindent

{\bf Acknowledgements}: The author would also like to thank the anonymous referees for their careful reading and useful comments that helped
to improve the manuscript.

The authors would also like to thank Jimmy Lamboley for useful discussions. 

I. Ftouhi and G. Paoli are supported by the Alexander von Humboldt Foundation through Alexander von Humboldt grants for Postdocs.

A. L. Masiello and G. Paoli are partially supported by GNAMPA of INdAM. 


\bibliographystyle{abbrv}
\bibliography{biblio.bib}

\end{document}